\documentclass[reqno,a4paper]{article}

\usepackage{amssymb,amsmath,epsfig,graphics,graphicx,color,subfigure,blindtext}

\definecolor{lightblue}{rgb}{0.22,0.45,0.70}
\usepackage[colorlinks=true,breaklinks=true,linkcolor=lightblue,citecolor=lightblue]{hyperref}
\usepackage{float}
\usepackage{tikz}
\usepackage{pgfplots}
\usetikzlibrary{arrows.meta}
\usepackage[font=small,skip=0pt]{caption}

\usepackage{multicol}
\usepackage{mathrsfs}
\usepackage{comment}
\usepackage{caption}

\usetikzlibrary{arrows,decorations.markings}
\pgfplotsset{compat=1.5}
\usepackage{pgfkeys}
\def\addlegendimage{\csname pgfplots@addlegendimage\endcsname}

\newtheorem{remark}{Remark}[section]
\newtheorem{lemma}{Lemma}[section]
\newtheorem{theorem}{Theorem}[section]
\newtheorem{proposition}{Proposition}[section]

\renewcommand{\O}{\Omega}

\newcommand{\ds}{\,\mbox{d}s}

\newcommand\E{{E}}
\newcommand\dE{{\partial \E}}

\newcommand\Vh{{\mathcal V}_h}
\def\qin{{\quad\hbox{in}\quad}}

\newcommand{\calS}{ \mathcal{S}}

\newcommand\R{\mathbb{R}}
\newcommand\N{\mathbb{N}}
\renewcommand{\P}{{\mathbb P}}  
\newcommand{\M}{ \mathbb{M}}
\def\wP{\widetilde{\P}}

\def\gb{{\bf g}}

\def\xb{{\bf x}}

\def\vb{{\bf v}}

\def\nb{{\bf n}}
\def\tb{{\bf t}}

\def\nbs \boldsymbol{n}
\def\0{\boldsymbol{0}}

\def\calD{\mathcal D}
\def\calS{\mathcal S}

\def\tvI{\widetilde{v}_{\rm I}}
\def\vI{v_{\rm I}}

\def\D{ \mathbf{D}}

\def\CT{\mathcal{T}}

\def\CT{\O}

\def\VE{V_k^{h}(\E)}
\def\Vh{V_k^{h}}

\def\VVE{\mathbb{V}_k^{h}(\E)}
\def\VVh{\mathbb{V}_k^{h}}

\newcommand\bn{{\bf n}}
\def\dn{\partial_{\bn_{\E}}}

\def\PiK{\Pi_{\E}^{\D,k}}
\def\PiK{\Pi_{\E}^{\D,k}}

\def\Oh {\left\{\CT_h \right\}_{h>0}}

\newcommand\bt{{\bf t}}

\newcommand{\cblue}[1]{\textcolor{blue}{#1}}

\newenvironment{proof}{\noindent{\it Proof.}}{\hfill$\square$}

\begin{document}
\title{$C^1$ virtual element methods on polygonal meshes with curved edges}
	
\author{L. Beir\~ao da Veiga\thanks{Dipartimento di Matematica e Applicazioni,
		Universit\`a degli Studi di Milano Bicocca, Via Roberto Cozzi 55, Milano, Italy and
		IMATI-CNR, Via Ferrata 1, Pavia, Italy. 
		E-mail: {\tt lourenco.beirao@unimib.it}.},\quad 
	D. Mora\thanks{GIMNAP, Departamento de Matem\'atica, Universidad
		del B\'io-B\'io, Concepci\'on, Chile and
		CI$^2$MA, Universidad de Concepci\'on, Concepci\'on, Chile.
		E-mail: {\tt dmora@ubiobio.cl}.},\quad 
	A. Silgado\thanks{Dipartamento di Matematica, Universit\`a degli Studi di Bari, via E. Orabona, 4, 70125, Bari, Italy. E-mail:
		{\tt alberth.silgado@uniba.it}.}}

	\date{}
	\maketitle
	
\begin{abstract}
In this work we design a novel $C^1$-conforming virtual element method of arbitrary order $k \geq 2$,
to solve the biharmonic problem on a domain with curved boundary and internal curved interfaces in two dimensions. 
By introducing a suitable stabilizing form, we develop a rigorous interpolation, stability and convergence analysis obtaining optimal error estimates in the energy norm. Finally, we validate the theoretical findings through numerical experiments.
\end{abstract}
\smallskip\noindent


	\maketitle
	\section{Introduction} 


Partial Differential Equations (PDEs) are fundamental in modeling several physical phenomena, 
such as fluid and solid mechanics, electromagnetic fields, heat conduction, among others. 
Although numerical schemes have significantly improved the ability to approximate the solution of PDEs, the complexity and challenge increase notably when these equations are posed on domains whose boundary is curved (or have internal curved interfaces). Such kinds of domains appear commonly in practical applications, including aerodynamics, biomedicine, geophysics, among others, where the geometries of interest do not conform to simple flat shapes. It is well known that standard numerical schemes that utilize meshes with straight edges approximate the curved interface/boundary of interest through a linear interpolation. This process introduces a geometry approximation error, which can significantly affect the analysis and the optimal rate of convergence, especially for schemes of higher-order (see for instance~\cite{Thomee73}).
In order to overcome this drawback, several alternatives have been developed. For instance, among these we can find the isoparametric finite element~\cite{CR72,Zlamal1973,Lenoir86}, where high-order polynomial approximations of the curvilinear boundary is required, together with a careful choice of the isoparametric nodes. Another interesting approach is the isogeometric analysis~\cite{HCB2005,BBCHS2006,CHB2009}, which can obtain the exact representation of computational (CAD) domains. 

Recently, the curved virtual element method \cite{BRV2019_M2AN} has been presented as an alternative to avoid both the inconveniences associated with the effect of geometric error and the approach of finite isoparametric methods mentioned above. The Virtual Element Method (abbreviated as VEM) is a recent technology introduced in the pioneering work~\cite{BBCMMR2013} (we also refer to~\cite{BBMR2014} for some computational implementation aspects) and belongs to the group of polytopal Galerkin methods for discretizing PDEs. These schemes have garnered significant attention in recent years  from the scientific community  due to their inherent versatility in handling complex geometries. In particular, since its introduction, the VEM has been employed to discretize a wide range of problems in continuous mechanics. We refer to~\cite{VEM-Book,BBMR2023_ActaNumer} for a current state of the art on VEM and different applications.

It is well known that to construct conforming numerical schemes for fourth-order problems (in primal form) is required that the discrete spaces have $C^1$-continuity, this fact renders the methods highly complex. In particular, for the two-dimensional case, by employing traditional finite element schemes, polynomials of high-order are required to create $C^1$-approximations (e.g. degree fourth and five for the Bell and Argyris triangles, respectively), increasing the number of degrees of freedom and hence the computational complexity; for further details, we refer to~\cite[Chap. 6, sect. 6.1]{ciarlet}.  In contrast, the authors in~\cite{BM13} (see also \cite{ABSV2016,ChM-camwa}), have been exploiting the ability of the VEM to construct  discrete spaces with high-regularity. In particular, they have developed $C^1$-conforming schemes based on the VEM for solving fourth-order problems by employing   
low-order polynomials and few degrees of freedom; it can be seen that the lowest polynomial degree is $k=2$, and the associated degrees of freedom are $3$ per element vertex (the function and its gradient values vertex). In addition, we mention that VEM has also been applied to general polyharmonic problems, see for instance~\cite{AMV2022}.

The applicability of the VEM to solve elliptic problems on domains with curves edges and internal interfaces in two dimensions was started in~\cite{BRV2019_M2AN}. More precisely, the authors have been developed $C^0$-VEMs of high-order to solve the Poisson problem. By exploiting the facts that the VEM does not need an explicit expression of the basis functions and that the spaces are defined directly in physical domain (without using reference elements), the curvilinear approach enables the definition of discrete spaces within curved elements (directly), allowing an exact representation of the domain of interest by relying solely on a careful selection of the degrees of freedom and using a piecewise regular parametrization of the domain interface/boundary.

Since its introduction the approach of curved VEM has been employed to discretize several problems. For instance, in~\cite{ABD2020,AHABW2020} have been developed curvilinear virtual elements with application to two dimensional solid
mechanics and contact problems.  In~\cite{DFLSSV2021_cmame,DFSV2022_camwa}, the curved VEM in mixed formulations have been developed in two and three dimensions. The approach also has been utilized for approximating solutions to the wave equation~\cite{DFMSV2022_JSC}.  Other approaches have been investigated in~\cite{FS2018_M2AN,BBMR2020_M3AS,BPP2022_cmame} for two and three dimensions.

On the other hand, in \cite{BLMRV2024_JSC} the authors have been designed a curved VEM by employing the nonconforming approach, which is based on the computation of a novel Ritz-Galerkin operator that is different from the standard $H^1$-projection typically used in the VEM framework.
 
Further, different polytopic schemes have been developed to address curved domains. Among these, we mention the hybrid high-order method for the Poisson~\cite{BD-P2018_JCOMP,Yemm2023} and the singularly perturbed fourth-order problems~\cite{DE2021_M2AN}; the unfitted hybrid high-order method~\cite{BCDE2021_SISC}; then the extended hybridizable discontinuous Galerkin method~\cite{GSKF2016_JSC} and the Trefftz-based finite element method~\cite{AORW2020_SISC}. 

In this work we are interested in solving the Biharmonic problem on domains with curved boundary.  Fourth-order problems arise in many physical phenomena, for instance, in plate bending problems, fluid flow problems (in stream function form), the Cahn--Hilliard phase-field model, micro-electromechanical systems, among others. Due to its importance and challenging nature, several approaches have been devoted to the design of numerical schemes to solve these problems; see for instance~\cite{ciarlet,CMN2021,MS2003,GH2009,BGGS2012}, where classical conforming/nonconforming finite element schemes and  $C^0$-IP methods have been developed. Additionally, several schemes and analyses addressing fourth-order problems by employing the $C^1$-conforming VE approach have been presented; see for instance~\cite{BM13,ChM-camwa,ABSV2016,MRV2018_M2AN,MRS2021-IMAJNA,G2021,KMBb2025}.

This paper is devoted to the design of $C^1$-conforming VEM of arbitrary order $k \geq 2$ on curved domains and the development of a rigorous interpolation, stability and convergence analysis, showing \emph{optimal} error estimates in the $H^2$-norm. In particular, the proposed approach generalizes the $C^1$-VE space to the case of curved boundaries and internal interfaces. When the curved boundary is actually straight, the VE space and degrees of freedom of the proposed method simplifies to the space and degrees of freedom described in~\cite{BM13}. However, the definition of the stabilizing form remains different (see below~\eqref{eq:stab:form} and Remark~\ref{Remark:new:stab}).

The construction of the  present scheme takes the steps from the approach developed in~\cite{BRV2019_M2AN} for the simpler $C^0$-VEM for the Poisson equation. An adequate selection of the degrees of freedom and a piecewise regular parametrization of the domain boundary allow the construction of a computable  $H^2$-VEM operator, and a computable load term in the virtual space. A rigorous interpolation, stability and convergence analysis is provided.

It is interesting to note that the analysis here presented cannot be accomplished by a standard extension of the arguments in \cite{BRV2019_M2AN}. The proving strategy of many important aspects needs to deviate considerably from the path laid in the mentioned article as we briefly outline here below. 

This difficulty is related to the $H^2$-conformity required for the discrete space and the need to handle the associated higher-order terms on curved edges. Another peculiar challenge of the present approach is related to the preservation of the kernel of the involved bilinear form which for the simpler second-order diffusion problem is guaranteed by construction (since the local kernel is given by constant functions), while here it fails to be valid. Indeed, in the present approach, the space $\VE$ does not contain the kernel of $a_{\E}(\cdot,\cdot)$, i.e., $\P_1(\E) \nsubseteq \VE$ (see Remarks~\ref{remark:non:poly:containing} and~\ref{Remark:new:stab}).
As a consequence, in order to prove stability and convergence of the scheme, we need to introduce a new stabilizing form, which is made of partial contributions of the degrees of freedom (the internals ones and the evaluations of the function and its gradient at the vertexes), plus integrals on edges involving tangential and normal derivatives of the virtual functions (see below~\eqref{eq:new:stab:form} and~\eqref{eq:stab:form}). This change is motivated by the fact that, the classical choice for the stabilizing form, i.e., the ``\emph{dofi-dofi}'' stabilization~\cite{BBCMMR2013,BM13}, is not guaranteed to satisfy the coercivity property in the space  $\VE+\P_1(\E)$. 

We also recall that, unlike in~\cite{BRV2019_M2AN}, the {\it maximum principle} does not hold for biharmonic functions in this setting. Nevertheless, the $H^2$-decomposition result, together with the arguments developed in this work, suffices to establish the coercivity of the stabilizing form (cf. Lemma~\ref{lemma:coercivity:S_E}) and, consequently, the coercivity of the bilinear form  $a_{E}(\cdot,\cdot)$ (cf. Proposition~\ref{prop:stability:coerc}). See also Remark~\ref{remark:error:estimate}.

Finally, in this contribution we define a specific error quantity, namely, $T^h(u)$ (cf.  equation~\eqref{eq:non-conformity}),  which accounts for the VEM approximation of the bilinear form $a_{E}(\cdot,\cdot)$. The challenge in handling this term in the current context is that the norm in which the discrete bilinear form is continuous is stronger than the norm in which we are able to derive interpolation error estimates, see Remark~\ref{remark:non-conformity}. Furthermore, since this term involves the sum space $\Vh + \P_k(\O_h)$, it is not easily handled by bridging the two norms with inverse inequalities.

\paragraph*{Preliminary notations and model problem.}
For any open bounded Lipschitz domain $\calD$ in $\R^2$, and for $t \geq 0$,
we will employ the usual notation for the Sobolev spaces $H^{t}(\calD)$, according to Reference \cite{AF2003}, with their corresponding seminorm  and norms  denoted by $|\cdot|_{t,\calD}$  and  $\|\cdot\|_{t,\calD}$, respectively.
In  the particular case $t=0$,  we  write  $L^2(\calD)$  instead of $H^{0}(\calD)$ and its 
standards $L^2$ inner-product and norm are denote by $(\cdot,\cdot)_{0,\calD}$  and  $\|\cdot\|_{0,\calD}$.  
Furthermore, we introduce the Sobolev spaces~$W^{s,\infty}(\calD)$, $s\in \N \cup \{0\}$,
of functions having weak derivatives up to order~$s$, which are bounded almost everywhere in~$\calD$.
The particular case~$s=0$ coincides with the space $L^\infty(\calD)$. As usual,
the spaces of non-integer order~$s > 0$, are constructed by interpolation theory.
The corresponding seminorm and norm are denote by $|\cdot|_{W^{s,\infty}(\calD)}$ and
$\|\cdot\|_{W^{s,\infty}(\calD)}$, respectively. Moreover, with the usual notation the symbols $\nabla v$, 
$\Delta v$, $\Delta^2 v$  and $\D^2 v$ denote the gradient, the Laplacian, the Bilaplacian and the 
Hessian matrix for a regular enough scalar function $v$.

From now on, $\O $ will denote a Lipschitz domain in $\R^2$ with (possibly) curved boundary~$\Gamma:=\partial\Omega$ 
and unit outward normal $\bn$. We consider the Biharmonic problem with homogeneous Dirichlet boundary conditions, which reads as: given $f \in L^2(\Omega)$, find $u: \O  \to  \R$ such that
	\begin{equation}\label{Bihar:model}
		\begin{cases}
			\Delta^2 u = f   &\text{in } \Omega,\\
			u = \partial_{\bn} u = 0       &\text{on } \Gamma.
		\end{cases}
	\end{equation}

A variational formulation of problem~\eqref{Bihar:model}  reads as follows: find $u \in V:= H_0^2(\O)$, such that 
 \begin{equation}\label{Bihar:weak}
 	a(u,v)= (f,v)_{0,\Omega} \qquad \forall v\in V,	
 \end{equation}
where the bilinear form is given by:
\begin{equation}\label{cont:form}
a(u,v) := (\D^2 u, \D^2 v)_{0,\Omega}\ \qquad \forall u,v\in V.
\end{equation}
By using that $\|\D^2 \cdot\|_{0,\O}$ is a norm equivalent to the  $H^2$-norm, we have that the bilinear form $a(\cdot,\cdot)$ is $V$-elliptic. Thus, problem~\eqref{Bihar:weak} is well posed as a consequence of the Lax-Milgram Theorem.

Now, we will consider additional notations and settings. For the domain $\O$, we assume that its 
boundary~$\Gamma$ is the union
a finite number of smooth curves $\{\Gamma_i\}_{i=1}^{N}$, i.e.,
\[
\Gamma= \bigcup_{i=1}^N\Gamma_i.
\]

Moreover, we consider the following assumption: 
\begin{description}
	\item[$\mathbf{(A0)}$] the boundary $\Gamma$ is Lipschitz, and each component curve $\Gamma_i$ of $\Gamma$ is sufficiently smooth. In particular, we assume that each $\Gamma_i$ is of class {$C^{k+1}$}, where $k \geq 2$ will be introduced later in Section~\ref{SECTION:VE:SCHEME}, as it is related to the polynomial degree of the numerical method.
\end{description}

Let  $I_i := [a_i, b_i] \subset \R$ be a closed interval of $\R $ and $\gamma_i \colon I_i \to \Gamma_i$ be a given {$C^{k+1}$}-regular and invertible parametrization of the curve $\Gamma_i$, for $i=1, \dots, N$. 

In what follows, as all parts  $\Gamma_i$ of $\partial \Omega$ will be handled similarly, for the sake of simplicity in notation, we will omit the index $i$ from the related maps and parameters.

\paragraph*{Outline.} The layout of this paper is as follows. In Section~\ref{SECTION:VE:SCHEME} we first introduce some notations and basic settings on the regular polygonal meshes, and afterwards describe the virtual element spaces together with their degrees of freedom, the discrete forms and the VE formulation. In Section~\ref{THEORIC:SECTION}, we provide a rigorous analysis for the proposed VE scheme. In particular, we establish interpolation, stability and convergence analysis for the method. 
Finally, in Section~\ref{SECTION:NUMERICAL:RESULT} of the article we present a set of numerical tests underlining the optimal behaviour of the scheme also from the practical perspective.
\setcounter{equation}{0}	
\section{The $C^1$ virtual element scheme on curved domains}\label{SECTION:VE:SCHEME}
In this section, we introduce a family of $C^1$-conforming VEMs  of arbitrary order $k \geq 2$ for the numerical approximation 
of problem \eqref{Bihar:weak} on curved domains. 
First, we introduce some ingredients to construct the discrete virtual scheme, 
such as,  notations on the polygonal decompositions, local and global virtual spaces
together with their corresponding degrees of freedom, polynomial projections, the
discrete bilinear forms and the load term approximation. We conclude this section presenting 
the discrete VE formulation.
  
\subsection{Notations and  basic settings on the polygonal decompositions}
\label{sec:dofs}

\paragraph{Mesh assumptions.}
Henceforth, we will denote by $\E$ a general polygon, and by $e$ a general edge of 
$\partial\E$ (possibly curved).  We denote by $h_\E$ the diameter of the element  $\E$ and by 
$h_e$ the size of the edge $e$, where the size of an edge is the distance 
between its two endpoints (see  below Remark~\ref{length:edges}). 
Moreover, we will denote by $x_e$ and $\xb_{\E}$ the midpoint of $e$ 
and the baricenter of $\E$, respectively. 
For each element $\E$, we denote by $N_\E$ the number 
of vertices of $\E$, by $N_e$ the number of edges of $\E$. We associate the outward unit normal 
vector $\bn_\E$ and the unit tangent vector  $\bt_{\E}$ to $\E$.
Although for a polygon the quantities $N_E$  and $N_e$ are the same, we prefer to denote them differently.

Let $\Oh$ be a sequence of decompositions of $\O$ into general 
non-overlapping polygons $\E$, where $h:=\max_{\E\in\CT_h}h_\E.$
For all $h$, we will consider the classical regularity assumptions on the decomposition $\CT_h$:
there exist a  positive constant $\rho$ independent of $h$, such that
\begin{description}
	\item [$\mathbf{(A1)}$] each element $E$ is star-shaped with respect to a ball $B_{\E}$  larger than or equal to $ \geq\, \rho \, h_E$, 
	\item [$\mathbf{(A2)}$] for each element $\E$ and any of its edges  (possibly curved),  the edge length satisfies $h_e\geq \rho \, h_E$.
\end{description}

We will write $a\lesssim b$ and  $a\gtrsim b$  instead of $a\leq Cb$ and  $a\geq Cb$, respectively, where $C$ is positive constant  independent of the decomposition~$\CT_h$. Sometimes we write $a\approx b$ when 
$a \lesssim b$ and~$b \lesssim a$ occur at the same time.  The involved constants will be explicitly  written only when necessary. 
Moreover, for the Sobolev norms on curves edges, we will employ the same notation and definition as in \cite{BRV2019_M2AN}.

Assumptions $\mathbf{(A1)}$--$\mathbf{(A2)}$ imply that each element $\E$ has a uniformly bounded number of edges. Moreover, it guarantees that the constants in the standard inverse estimates and the forthcoming trace   inequalities are uniformly bounded (see below Lemmas~\ref{lemma:trace:BRV} and~\ref{lemma:scaled:trace}).

\begin{remark}\label{length:edges}
Let $L_e =\int_{I_e}\|\gamma'(t)\|$ be the length of a curved edge.
Since $\gamma$ and~$\gamma^{-1}$ are fixed once and for all, and both are of class~$W^{1,\infty}$, 
we have that the values $h_e$ and $L_e$ are comparable.
Therefore in what follows, by simplicity we will refer to both quantities as ``length'' of $e$.
\end{remark}

\paragraph{Related polynomial spaces.}
Let~$\P_{\ell}(\E)$, $\ell\in \N$, be the space of polynomials of maximum degree~$\ell$ over each 
element~$\E$, with the convention $\P_{-2}(\E)=\P_{-1}(\E) = \{0\}$.
We introduce a basis for the space~$\P_{\ell}(\E)$ given by 
the set of shifted and scaled monomials:
\begin{equation*}
\mathbb{M}_{\ell}(\E) =
\left\{
\left(\dfrac{\xb-\xb_{\E}}{h_{\E}}\right)^{\boldsymbol{\alpha}}
\; \forall  \boldsymbol{\alpha} \in \N^2, \; \vert \alpha \vert \le \ell,\;\;
\forall \xb \in \E \right\},
\end{equation*}
where $\boldsymbol{\alpha}$ denotes a multi-index~$\boldsymbol{\alpha}=(\alpha_1,\alpha_2)$.

Similarly, we consider the space of polynomials of maximum degree~$\ell$ over the interval~$I_e$, 
which is defined by $\P_\ell(I_e)$, with $\ell\in \N$. Given~${x}_{I_e}$ and~$h_{I_e}$ the midpoint 
and  the length of~$I_e$, we introduce a basis for the polynomial space~$\P_\ell(I_e)$ given by the following 
set of shifted  and scaled  monomials:
\begin{equation*}
\mathbb{M}_{\ell}(I_e) =
\left\{
\left(\dfrac{{x}-{x}_{I_e}}{h_{I_e}}\right)^{\alpha}
\; \forall  \alpha \in \N, \;  \alpha \le \ell,\;\;
\forall x \in I_e \right\}.
\end{equation*}

Since~$\gamma_e: I_e \to e$ denotes the parametrization of the edge~$e$,
we consider the following mapped polynomial space:
\begin{equation}\label{def:poli}
\widetilde{\P}_{\ell}(e)      = \big\{\widetilde{q_{\ell}} = \widehat{q_{\ell}}\, \circ \gamma_e^{-1}:\ \widehat{q_{\ell}}\in \P_{\ell}(I_e)\big\}.
\end{equation}
Clearly, whenever $e$ is straight (in which case we always take $\gamma$ as an affine mapping), the space $\widetilde{\P}_{\ell}(e)$ correspond to the standard polynomial space $\P_{\ell}(e)$. 

Finally, we define the piecewise $\ell$-order polynomial space by:
\begin{equation*}
	\P_{\ell}(\CT_h) := \{ q \in L^2(\O): q|_{\E} \in \P_{\ell}(\E) \quad\forall \E \in \CT_h \}. 
\end{equation*} 
\subsection{Virtual element spaces and their degrees of freedom}

In this subsection, we introduce a family of conforming virtual element spaces on curved domains and equip these spaces with suitable sets of degrees of freedom. For simplicity of exposition, in the following we will assume that every element $\E$ has at most one curved edge lying on $\Gamma$ (hence, only one edge of $\E$ will be curved and the remaining will be straight); the extension to the case with more curved edges is trivial. Moreover, we assume that each curved edge lies on only one regular curve, i.e, $e \subseteq \Gamma_i$. This last condition is mandatory for the approach followed in this work.

Let $E \in \Omega_h$ with $\partial E = \bigcup_{i=1}^{N_E} e_i$, where $e_1 \subset \Gamma$, and $e_2, \cdots, e_{N_E}$ are straight segments. Thus, for each integer $k \geq 2$ we  introduce the index $r:=\max\{k,3\}$ and the local virtual space on the curved element $E$:
\begin{multline}\label{local:virtual:space}
\Vh(\E) := \Bigl\{ v_h \in H^2(\E):   
		\Delta^2 v_h \in \P_{k-4}(\E),\ v_h|_{e_1} \in \widetilde{\P}_{r}(e_1),\: \partial_{\nb_{\E}} v_h|_{e_1} \in\widetilde{\P}_{k-1}(e_1), \Bigr. \\
	\Bigl. v_h|_{e_i}\in \P_r(e_i), \: \partial_{\nb_{\E}} v_h|_{e_i} \in\P_{k-1}(e_i) \quad \text{for $i=2, \dots, N_E$} \Bigr\}.
\end{multline}

\begin{remark}\label{remark:non:poly:containing}
From definition \eqref{local:virtual:space} it is clear that if  the polygon $\E$ has at least a curved edge, then $\P_0(E) \subseteq \VE$ but $\P_k(E) \nsubseteq \VE$ for $k \ge 1$. The property $\P_0(E) \subseteq \VE$ was critical in the analysis of \cite{BRV2019_M2AN} since $\P_0(E)$ represents the kernel of the associated local bilinear form $(\nabla \cdot, \nabla \cdot)_{0,E}$ for the simpler Poisson equation. Therefore, such property was guaranteeing that the kernel of the original local bilinear form was contained in the virtual space.
For the present biharmonic problem, the local kernel of the involved bilinear form is $\P_1(E)$, which is not contained in $\VE$ for curved elements.
This makes it more complex to prove the stability properties of the discrete bilinear form because the kernel clearly plays a key role in the related ``energy'' equivalence properties (from a more technical standpoint, the importance of $\P_1(E)$ is clear from the proof of Proposition \ref{prop:stability:coerc}, which in turn needs Lemma \ref{lemma:coercivity:S_E}, showing coercivity on the sum space $\VE + \P_1(E)$). 
In particular, we need to be more careful in constructing the discrete approximated bilinear form $a^h(\cdot,\cdot)$, especially with the design of the stabilizing part $\mathcal{S}_{\E}(\cdot,\cdot)$; see below~\eqref{eq:new:stab:form}.
Additional explanations can be found in Remark \ref{Remark:new:stab}.		
\end{remark}

For each $v_h \in\VE$, we define the following sets  of linear functionals $\boldsymbol{D}$ (split into boundary operators $\boldsymbol{D^{\partial}}$ and internal ones $\boldsymbol{D^o}$):
\begin{itemize}
	\item $\boldsymbol{D^{\partial}_{\!I}}$: the values of $v_h$ at the vertexes $\vb_i$  for $i=1, \dots, N_E$ of the element $E$;
	\item $\boldsymbol{D^{\partial}_{\!I\!I}}$: the values of $ h_{\vb_i}\nabla v_h$ at the vertexes $\vb_i$  for $i=1, \dots, N_E$ of the element $E$;
	\item $\boldsymbol{D^{\partial}_{\!I\!I\!I}}$: the values of $v_h$ at $k_e=\max\{0,k-3\}$ internal points of the $(k+1)$-point  Gauss-Lobatto quadrature rule of every straight edge $e_2, \dots, e_{N_E} \in \partial E$;
	\item $\boldsymbol{D^{\partial}_{\!I\!V}}$: the values of $v_h$ at $k_e=\max\{0,k-3\}$ internal points of $e_1$ that are the images through $\gamma$ of the $k_e$ internal points of the  $(k+1)$-point  Gauss-Lobatto quadrature on $I_{e_1}$;
	\item $\boldsymbol{D^{\partial}_{\!V}}$: the values of $h_e \partial_{\nb_{\E}}v_h$ at $k_{\bn}=k-2$ internal points of the $k$-point  Gauss-Lobatto quadrature rule of every straight edge $e_2, \dots, e_{N_E} \in \partial E$;
	\item $\boldsymbol{D^{\partial}_{\!V\!I}}$: the values of $h_e\partial_{\nb_{\E}}v_h$ at $k_{\bn}=k-2$ internal points of $e_1$ that are the images through $\gamma$ of the $k_{\bn}$ internal points of the  $k$-point  Gauss-Lobatto quadrature on $I_{e_1}$,
	\item $\boldsymbol{D^o}$: for $k\geq 4 $, the internal moments of $v_h$ against the scaled monomials $\{m_j\}^{(k-3)(k-2)/2}_{j=1} \in \M_{k-4}(\E)$, i.e.,
	\[
	\boldsymbol{D}_j^{\boldsymbol{o}}(v_h):= h^{-2}_{\E} \int_E v_h \, m_j \, {\rm d}\E,
	\]
\end{itemize}
where $h_{\vb_i}$ corresponds to the average of the diameters corresponding to the elements with $\vb_i$ as a vertex.
A visualization of these DoFs on a curved pentagon can be seen in Figure~\ref{fig:dofs:loc}, for the cases $k=3$ and $k=4$.

\begin{figure}[!h]
	\centering
	\begin{tikzpicture}[scale=0.5]
		\draw[line width=0.35mm, black ] (0,0) -- (8,0)--(10,5);	
		\draw[line width=0.35mm, red ] (10,5) arc (0:102:4.1cm) node[midway]{$\blacksquare$};
		\draw[line width=0.35mm, black ] (5,9)--(1,5)--(0,0);	
		\draw [fill] (0,0) circle (0.1);
		\draw [thick] (0,0) circle (0.4);
		\draw [fill] (8,0) circle (0.1);
		\draw [thick] (8,0) circle (0.4);  
		\draw [fill] (10,5) circle (0.1); 
		\draw [thick] (10,5) circle (0.4);  
		\draw [fill] (5,9) circle (0.1);
		\draw [thick] (5,9) circle (0.4);  	
		\draw [fill] (1,5) circle (0.1);
		\draw [thick] (1,5) circle (0.4); 
		 \draw[] (0,0) -- (8,0) node[midway] {$\cblue{\bigstar}$};
		 \draw[] (5,9) -- (1,5) node[midway] {$\cblue{\bigstar}$};
		 \draw[] (0,0) -- (1,5) node[midway] {$\cblue{\bigstar}$};
		 \draw[] (8,0) -- (10,5) node[midway] {$\cblue{\bigstar}$};
	
		\draw[line width=0.35mm, black ] (15,0) -- (23,0)--(25,5);	
		\draw[line width=0.35mm, red] (25,5) arc (0:102:4.1cm) node[pos=0.25]{$\blacksquare$};
		\draw[line width=0.35mm, red] (25,5) arc (0:102:4.1cm) node[midway]{$\blacktriangle$};
		\draw[line width=0.35mm, red] (25,5) arc (0:102:4.1cm) node[pos=0.75]{$\blacksquare$};
		\draw[line width=0.35mm, black ] (20,9)--(16,5)--(15,0);	
		\draw [fill] (15,0) circle (0.1);
		\draw [thick] (15,0) circle (0.4);
		\draw [fill] (23,0) circle (0.1);
		\draw [thick] (23,0) circle (0.4);
		\draw [fill] (25,5) circle (0.1);
		\draw [thick] (25,5) circle (0.4); 
		\draw [fill] (20,9) circle (0.1);
		\draw [thick] (20,9) circle (0.4);  	
		\draw [fill] (16,5) circle (0.1);
		\draw [thick] (16,5) circle (0.4);     
		\draw[] (15,0) -- (23,0) node[pos=0.25] {$\cblue{\bigstar}$} node[midway] {${\blacklozenge}$} node[pos=0.75] {$\cblue{\bigstar}$};
		\draw[] (23,0) -- (25,5) node[pos=0.25] {$\cblue{\bigstar}$} node[midway] {${\blacklozenge}$} node[pos=0.75] {$\cblue{\bigstar}$};
		\draw[] (20,9) -- (16,5) node[pos=0.25]  {$\cblue{\bigstar}$}
		 node[midway]  {${\blacklozenge}$} 
		 node[pos=0.75]  {$\cblue{\bigstar}$};
	    \draw[] (16,5) -- (15,0) node[pos=0.25] {$\cblue{\bigstar}$}		
         node[midway] {${\blacklozenge}$}		
       	 node[pos=0.75] {$\cblue{\bigstar}$};
       	 
       	\draw[green, thick] (20.2,4.2) circle (0.3); 		
	\end{tikzpicture}
	\vspace*{1.5mm}
	\caption{DoFs for $k=3$ (left) and $k=4$ (right). We denote $\boldsymbol{D^{\partial}_{\!I}}$ with the dots, $\boldsymbol{D^{\partial}_{\!I\!I}}$ with the circles, $\boldsymbol{D^{\partial}_{\!I\!I\!I}}$ with  the black diamonds, $\boldsymbol{D^{\partial}_{\!I\!V}}$ with the red triangles, $\boldsymbol{D^{\partial}_{\!V}}$ with  the blue stars, $\boldsymbol{D^{\partial}_{\!V\!I}}$ with the red squares, and $\boldsymbol{D^o}$ with the green circle.}
	\label{fig:dofs:loc}
\end{figure}
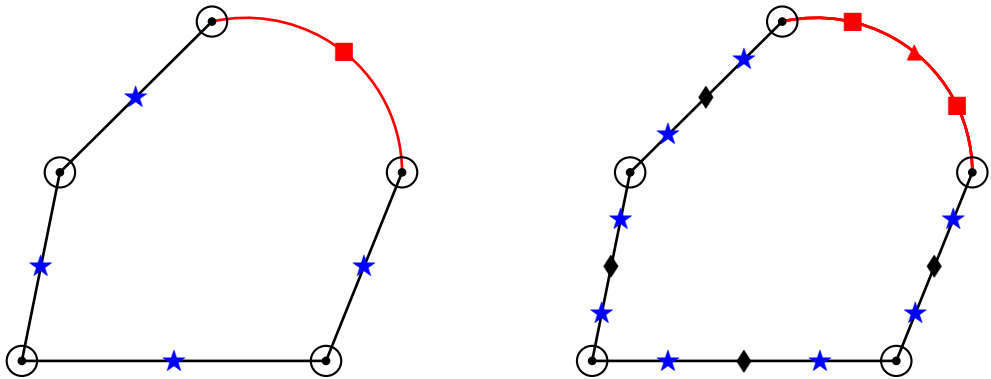

For all $ \E \in \O_h$, we have that the dimension of the local space $\VE$ is 
\begin{equation}\label{eq:dimension:VE}
	\dim(\Vh(\E)) = (3  +k_e +k_{\bn})N_e + \frac{(k-3)(k-2)}{2}.
\end{equation}

The following result generalizes \cite[Proposition 4.2]{BM13} to the case of polygons with curved
edges. Although the proof is quite predictable, we prefer to report it in detail for completeness.

\begin{theorem}
The sets of linear functionals $\boldsymbol{D}$ are a set of degrees of freedom for the space~$\VE$.
\end{theorem}
\begin{proof}
We observe that the cardinal of sets $\boldsymbol{D}$ is equal to the dimension of the space $ \VE$ (cf. \eqref{local:virtual:space} and  \eqref{eq:dimension:VE}), therefore it is sufficient to show that any virtual function with vanishing DoFs is the zero function. 
Using integration by parts, we have 
\begin{equation*}
\int_{\E} |\Delta v_h|^2 \: {\rm d}\E= \int_{\E} v_h\Delta^2 v_h  \:{\rm d}\E+ \int_{\partial\E} \partial_{\bn_\E}
 v_h\Delta v_h \:{\rm d}s- \int_{\partial\E} v_h \partial_{\bn_{\E}}\Delta v_h \: {\rm d}s. 
\end{equation*}

Following standard VEM arguments, the key point becomes to prove that $\boldsymbol{D^{\partial}_{i}}(v_h) = \mathbf{0}$, with $\boldsymbol{i}\in \{\boldsymbol{\!I}, \,\boldsymbol{\!I\!I}, \,\boldsymbol{\!I\!V},\, \boldsymbol{\!V\!I}\}$ imply that $v_h|_{e_1} \equiv 0$ and $\partial_{\bn_{\E}} v_h|_{e_1}\equiv 0$ (since for the straight edges and the volume part the proof is standard).
We first analyze the cases $k=2$ and $k=3$. 

For the case $k=2$, we have $r=\max\{3,k\}=3$. 
Thus, $v_h|_{e_1} \in \wP_3(e_1)$ and $\partial_{\nb_{\E}}v_h|_{e_1} \in \wP_1(e_1)$. 
We denote by $\vb_{e_1}^{{\rm i}}$ and $\vb_{e_1}^{{\rm f}}$ the initial and end points of $e_1$, respectively. 
Then, from the definition of normal derivative and $\boldsymbol{D^{\partial}_{\!I\!I}}$ we have 
\begin{equation*}
\partial_{\bn_{\E}} v_h(\vb_{e_1}^{{\rm i}})= \partial_{\bn_{\E}} v_h(\vb_{e_1}^{{\rm f}}) =0,   	
\end{equation*}
thus $\partial_{\bn_{\E}} v_h|_{e_1} \equiv 0$. Now, since $v_h|_{e_1} \in \wP_3(e_1)$, 
from definition of~\eqref{def:poli},  
we have
$$
v_h(\gamma({t})) =  \widetilde{q_3}(\gamma({t})) = \widehat{q_3}(t)  \qquad \forall t \in I_{e_1}, 
$$
where $\widehat{q_3} \in \P_3(I_{e_1})$. {Then, by applying the chain rule to functions of two variables, we obtain 
$$\frac{d\widehat{q_3}(t)}{dt}= \nabla v_h(\gamma(t)) \cdot \frac{ d\gamma(t)}{dt},$$ 
where $\nabla v_h$ is the gradient of $v_h$ and $\frac{d\gamma(t)}{dt}$ denotes the derivatives vector respect with the parameter $t$ of the $\gamma$-components.}

Now, letting $\{H_1,H_2,H_3,H_4\}$ be the four polynomials 
of the Hermite basis, then  by employing the above facts, we get
\begin{equation*}
\begin{split}
\widehat{q_3}(t) &= \widehat{q_3}(a_1)H_1(t) + {\frac{d\widehat{q_3}(a_1)}{dt}H_2(t)}
+ \widehat{q_3}(b_1)H_3(t) + { \frac{\widehat{q_3}(b_1)}{dt}H_4(t)}\\
&= v_h(\gamma({a_1}))H_1(t)+  {\nabla v_h(\gamma(a_1)) \cdot \frac{ d\gamma(a_1)}{dt}} H_2(t) + v_h(\gamma(b_1))H_3(t) + {\nabla v_h(\gamma(b_1)) \cdot \frac{ d\gamma(b_1)}{dt}} H_4(t)\\
&= v_h(\gamma({a_1}))H_1(t)+  \nabla v_h(\vb_{e_1}^{{\rm i}}) \cdot \frac{ d\gamma(a_1)}{dt} H_2(t) + v_h(\gamma(b_1))H_3(t) + \nabla v_h(\vb_{e_1}^{{\rm f}}) \cdot \frac{ d\gamma(b_1)}{dt} H_4(t),
\end{split}
\end{equation*} 
where we have used that $\gamma({a_1})=\vb_{e_1}^{{\rm i}}$ and $\gamma({b_1})=\vb_{e_1}^{{\rm f}}$. Therefore, by employing that $\boldsymbol{D^{\partial}_{\!I}} (v_h)=\boldsymbol{D^{\partial}_{\!I\!I}} (v_h)=\0$, it implies $\widehat{q_3} = 0$, and thus $v_h|_{e_1} \equiv 0$.

For the case $k=3$, we have $v_h|_{e_1} \in \wP_3(e_1)$ and $\partial_{\nb_{\E}}v_h|_{e_1} \in \wP_2(e_1)$. 
Now, let  $\{t_j\}_{j=1}^{3}$ represent the nodes of the  $3$-point  Gauss-Lobatto quadrature on $I_{e_1}$.
Then, by using the definition~\eqref{def:poli} and $\boldsymbol{D^{\partial}_{\!V\!I}}(v_h)=0$, we have
\begin{equation*}
\begin{split}
0=\partial_{\bn_\E} v_h(\gamma({t_2})) =  \widetilde{q_2}(\gamma({t_2})) = \widehat{q_2}(t_2),
\end{split}
\end{equation*}
where $\widehat{q_2} \in \P_2(I_{e_1})$. Moreover, as in the case $k=2$, we have $\partial_{\bn_\E} v_h(\vb_{e_1}^{{\rm i}})= \partial_{\bn_\E} v_h(\vb_{e_1}^{{\rm f}}) =0$,
and using $\gamma({a_1})=\vb_{e_1}^{{\rm i}}$, $\gamma({b_1})=\vb_{e_1}^{{\rm f}}$, we conclude $\widehat{q_2}(t_i)=0$, for $i=1,2,3$. Therefore $\widehat{q_2}\equiv 0$, that implies $\partial_{\bn_\E} v_h|_{e_1}\equiv 0$. The fact $v_h|_{e_1} \equiv 0$ follows exactly as in case $k=2$.

So far we have shown $v_h|_{e_1} \equiv 0$ and $\partial_{\bn_{\E}} v_h|_{e_1}\equiv 0$  for $k=2$ and $k=3.$
Now, we will consider the case $k \geq 4$. In this case, we have $r=k$, so $v_h|_{e_1}\in \wP_{k}(e_1)$ 
and $\partial_{\bn_\E} v_h|_{e_1}\in \wP_{k-1}(e_1)$. By definition, there exists
$\widehat{q}_k \in \P_k(I_{e_1})$, such that  
\begin{equation*}
v_h(\gamma({t})) =  \widetilde{q_k}(\gamma({t})) = \widehat{q_k}(t)  \qquad \forall t \in I_{e_1}, 	
\end{equation*} 

Let $\{f_1,f_2,\ldots,f_{n_k}\}$ be the canonical basis associated to the sets 
$\boldsymbol{D^{\partial}_{\!I}}, \boldsymbol{D^{\partial}_{\!I\!I}}$ and 
$\boldsymbol{D^{\partial}_{\!I\!V}}$, where $n_k = \dim(\P_k(I_{e_1}))$, such that the canonical functions $f_1,f_2,f_{n_k-1}$ and $f_{n_k}$ are associated to the canonical coefficients $\widehat{q_k}(a_1),$ 
${\frac{d\widehat{q_k}(a_1)}{dt}},$ $\widehat{q_k}(b_1)$ and ${ \frac{d\widehat{q_k}(b_1)}{dt}}$, respectively. Thus, as in the case $k=2$, for all $t \in I_{e_1}$, from the chain rule we have
 \begin{equation*}
	\begin{split}
		\widehat{q_k}(t) &= \widehat{q_k}(a_1)f_1(t) +{ \frac{d\widehat{q_k}(a_1)}{dt}}f_2(t) +
		\alpha_{3}f_3(t) + \cdots+ \alpha_{n_k-2}f_{n_k-2}(t)+\widehat{q_k}(b_1)f_{n_k-1}(t)+ {  \frac{d\widehat{q_k}(b_1)}{dt}}f_{n_k}(t)\\
		&= v_h(\gamma({a_1}))f_1(t)+ { \nabla v_h(\gamma(a_1)) \cdot \frac{ d\gamma(a_1)}{dt}}+ \alpha_{3}f_3(t) + \cdots+ \alpha_{n_k-2}f_{n_k-2}(t)\\
		& \qquad \quad + v_h(\gamma(b_1))f_{n_k-1}(t) + { \nabla v_h(\gamma(b_1)) \cdot \frac{ d\gamma(b_1)}{dt}} f_{n_k}(t).
	\end{split}
\end{equation*}  

Thus, by using $\boldsymbol{D^{\partial}_{\!I}}(v_h)=\boldsymbol{D^{\partial}_{\!I\!I}}(v_h)=\boldsymbol{D^{\partial}_{\!I\!V}}(v_h)=\0$, we obtain that $\widehat{q_k} \equiv 0$, hence $v_h|_{e_1} \equiv 0$. The fact $\partial_{\bn} v_h|_{e_1}\equiv 0$ follows from the same arguments used in the case $k=3$.
\end{proof}

The global virtual element space is constructed by combining the local spaces $\VE$ 
and incorporating the homogeneous Dirichlet boundary conditions, i.e., for every decomposition $\CT_h$ into polygons $\E$, we define   
\begin{equation}\label{glob:VEMspace}
\Vh = \{ v_h \in V : \: v_h|_{\E} \in \VE \quad \forall \E \in \CT_h  \},
\end{equation}
with the obvious associated sets of global degrees of freedom.

We notice that the unisolvency of global DoFs follows from the unisolvency of the local degrees of freedom and the definition of the local and global VE spaces.

\subsection{Polynomial projectors and discrete forms}

In this subsection we introduce some polynomial projections and the discrete bilinear
form approximating the continuous form~\eqref{cont:form}.  Moreover, we propose an approximation for the load term.

Note that the practical implementation of the projection operators introduced below requires to compute integrals of known functions (polynomials, mapped polynomials or data) on curved edges and on curved polygons. We accomplish these either by standard 1D quadrature or by suitable generalized rules, see for instance~\cite{SV2007,SV2009,CS2021}.

As usual, we decompose into local contributions the bilinear form $a(\cdot, \, \cdot)$ defined in~\eqref{cont:form} as follows 
\begin{equation*}
	a (u,v) = \sum_{E \in \Omega_h} a_E (u,v) \qquad  \forall u, v \in {V}.
\end{equation*}

\paragraph{Polynomial projections.} 
Next, we shall construct a VEM $H^2$-projection.  Indeed, for each polygon $\E$,  we define the  projector
$\PiK\colon H^2(\E)  \to \P_k(\E),$ as the solution of the local problems:
	\begin{equation}\label{Ritz:operator}
	\left\{ 	\begin{array}{ll}
	a_{\E}( v_h-\PiK v_h, q_k ) \qquad \forall  v_h \in H^2(\E) \quad \text{and} \quad q_k \in \mathbb{P}_k(\E),\\[2ex]
\Pi^0_\dE(v_h-\PiK v_h)=0, \quad 
\Pi^0_\dE(\nabla (v_h -\PiK v_h))=0,
	\end{array} \right.
\end{equation}
where  the operator $\Pi^0_\dE: H^2(\E) \to \P_0(\E)$ is given by:
\begin{equation}\label{averag:operator}
\Pi^0_\dE\varphi := |\dE|^{-1} \int_{\dE} \varphi  \ds.
\end{equation}

Furthermore, for each $m \in \N \cup \{0\}$ we consider the usual $L^2$-projection 
$\Pi_{\E}^{m}: L^2(\E)  \to \P_{m}(\E)$, defined by 
\begin{equation}\label{L2:proy:m}
	\int_{\E} p_m(v-  \Pi_{\E}^{m} v) \:{\rm d}\E =0  	\qquad 	\forall  v \in L^2(\E) \quad \text{and} 
	\quad\forall p_m\in\P_{m}(\E).
\end{equation}

The following result is easy to check by standard VEM arguments (see for instance~\cite{BM13}).
\begin{proposition}
For each $k \geq 2$ the operator $\PiK \colon \VE  \to \P_k(\E)$  defined in \eqref{Ritz:operator}
is computable from the degrees of freedom $\boldsymbol{D}$. Moreover, for $k \geq 4$ the  
$L^2$-projection  $\Pi_{\E}^{k-4}\colon \VE  \to \P_{k-4}(\E)$ defined in~\eqref{L2:proy:m} is computable from the 
degrees of freedom $\boldsymbol{D^{o}}$.
\end{proposition}

\paragraph{Discrete bilinear forms.} The next step consists in constructing a computable approximation of the bilinear form~\eqref{cont:form} with the tools given above. Indeed, first we define the following local bilinear form. Recalling Remark~\ref{remark:non:poly:containing}, we consider the sum space 
\begin{equation}\label{keyspace}
\VVE:= \VE + \P_k(\E). 
\end{equation}

Then, we set the following  bilinear form
\begin{equation*} 
a^h_E\colon \VVE \times \VVE \to \R,
\end{equation*}
approximating the local continuous form $a_E(\cdot, \, \cdot)$, and defined by
\begin{equation*}
	a^h_E(v_h,w_h) := a_E \big(\PiK v_h, \PiK w_h \big) + \mathcal{S}_E \big((I - \PiK) v_h, (I -\PiK) w_h \big) \qquad \forall v_h,w_h \in \VVE,
\end{equation*} 
where the form $\mathcal{S}_E \colon \VVE \times  \VVE \to \R$ is an adequate symmetric stabilizing bilinear form, which is computable by using the degrees of freedom $\boldsymbol{D}$.

In the present approach we need to introduce the following new stabilizing form (see Remarks 
\ref{remark:non:poly:containing} and \ref{Remark:new:stab}):
\begin{equation}\label{eq:new:stab:form}
	\begin{split}
\mathcal{S}_E ( v_h, \,  w_h) &:= h^{-2}_{\E}  \sum_{j=1}^{(k-3)(k-2)/2} \boldsymbol{D}^{\boldsymbol{\circ}}_j( v_h)\boldsymbol{D}^{\boldsymbol{\circ}}_j( w_h)
+\sum_{i=1}^{3N_{\E} } h^{-2}_{\vb_i}  \boldsymbol{D}^{\boldsymbol{\partial}, \vb}_i( v_h)\boldsymbol{D}^{\boldsymbol{\partial},\vb}_i( w_h)\\
&\quad + h_{\E} \sum_{i=1}^{N_e} \big( (\partial_{\tb_{\E}} (\partial_{\bn_\E} v_h),\partial_{\tb_{\E}} (\partial_{\bn_\E} w_h))_{0,e_i}  +(\partial^2_{\bt_\E} v_h,\partial^2_{\tb_{\E}}  w_h)_{0,e_i}\big)  
\qquad \forall  v_h,  w_h\in \VVE,		
	\end{split}
	\end{equation} 
where $\boldsymbol{D}^{\boldsymbol{\partial},\vb}$ is the vector of DoFs associated with the boundary DoFs 
$\boldsymbol{D^{\partial}_{\!I}}$ and $\boldsymbol{D^{\partial}_{\!I\!I}}$. 

We observe that this bilinear form  can be rewritten as the sum of three terms:  the first, $\mathcal{S}^{\boldsymbol{\circ}}_E$, involving the internal DoFs, the second $\mathcal{S}^{\boldsymbol{\partial},\vb}_E$  involving the vertex boundary DoFs, and the third $\mathcal{S}^{\boldsymbol{\partial},e}_E $ involving edges integrals, as follows
\begin{equation}\label{eq:stab:form}
\begin{split}
\mathcal{S}_E ( v_h, \,  w_h) = \mathcal{S}^{\boldsymbol{\circ}}_{\E} ( v_h, \,  w_h) + \mathcal{S}^{\boldsymbol{\partial},\vb}_{\E} ( v_h, \,  w_h)  + \mathcal{S}^{\boldsymbol{\partial},e}_{\E} ( v_h, \,  w_h) 
  \qquad \forall  v_h,   w_h \in \VVE,
\end{split}
\end{equation}
where the space $\VVE$ is defined in \eqref{keyspace}.

Finally we introduce the global bilinear form $a^h:\VVh\times\VVh \to \R$ given by
\begin{equation}\label{global:ah}
	a^h(u_h,v_h) = \sum_{E \in \CT_h} a^h_E (u_h,v_h) \qquad  \forall u_h, v_h \in \VVh,
\end{equation}
where global space $\VVh$ is defined by
\begin{equation*}
\VVh:=  \Vh + \P_k(\O_h). 	
\end{equation*}

\begin{remark}\label{Remark:new:stab}
The new stabilizing form $\mathcal{S}_{\E}(\cdot, \cdot)$ defined in~\eqref{eq:new:stab:form} is fully computable by using the DoFs $\boldsymbol{D}$. Furthermore, differently from more standard choices, such as the ``\emph{dofi-dofi}'' stabilization~\cite{BBCMMR2013,BM13}, the form \eqref{eq:new:stab:form} satisfies the coercivity property also in the sum space $\VE+\P_1(\E)$  (see Lemma~\ref{lemma:coercivity:S_E}).
Such sum space is complex to handle since the restriction to curved edges of $\P_1(\E)$ may lead to very general functions; thus standard VEM arguments based on inverse estimates on edges become very challenging, if not impossible, for the space $\VE+\P_1(\E)$.
Some intuition on the difficulties in proving the coercivity property for the ``\emph{dofi-dofi}'' stabilization can be gained by the following observation.

Letting $\eta = \textrm{dim}(\VE)$, the standard ``\emph{dofi-dofi}'' stabilization is essentially a scaled and squared norm on $\| \textrm{DOF}(v_h) \|$, ${v}_h \in \VE$, where $\textrm{DOF} \in {\mathbb R}^\eta$ represent the vector collecting all the local DoF values. Therefore $\textrm{DOF}$ is an application from an $\eta$ dimensional space into ${\mathbb R}^\eta$. When extended to $\VE + \P_1(E)$ (which has a dimension higher than $\eta$ since for curved elements $\P_1(E) \nsubseteq\VE$) such application has no hope of being injective; thus $\| \textrm{DOF}({v}_h) \|$ will not be a norm on $\VE + \P_1(E)$. By inverse estimates and scaling arguments on edges, the novel form $\mathcal{S}_{\E}(\cdot, \cdot)$ is indeed equivalent to the dofi-dofi on $\VE$, but the two stabilizing forms may not equivalent on larger spaces such as $ \VE+\P_1(E)$.
\end{remark}

\paragraph{Load term approximation.}
Here we shall define the discrete load term. 
As in \cite{BM13}, we will consider different approximations for the load term, 
depending on the polynomial degree $k$. Indeed, first we will introduce some useful notations:
\begin{equation*}
 f_h|_{\E} :=\begin{cases}
\Pi_{\E}^{k-2} f &\text{if  $ \, k =2,3$},\\
\Pi_{\E}^{1} f &\text{if  $ \, k =4$},\\
\Pi_{\E}^{k-4} f &\text{if  $ \, k >4$},
	\end{cases}
\end{equation*}
\begin{equation*}
\widehat{\Pi}_\E^1 v_h :=\begin{cases}
 \widehat{v_h} + (k-2)(\xb-\xb_{\E}) \cdot  \widehat{\nabla v_h} &\text{if  $\, k =2,3$},\\
\Pi_{\E}^{0} v_h + (\xb-\xb_{\E}) \cdot  \widehat{\nabla v_h} &\text{if  $ \, k =4$},\\
	\end{cases}
\end{equation*}
where 
\begin{equation*}
\widehat{w_h} = \frac{1}{N_E}\sum_{i=1}^{N_E} w_h(\vb_i), \quad \text{with  $\vb_i, 1 \leq i \leq N_\E$ are the vertices of  $\E$.}
\end{equation*}

Then, the right-hand side is defined by
\begin{equation}\label{global:load:term}
\langle f_h, v_h \rangle :=\begin{cases}
\sum_{\E \in \CT_h}	(f_h, \widehat{\Pi}_\E^1 v_h)_{0,\E} & \text{if  $\, k = 2,3,4$},\\
\sum_{\E \in \CT_h}	(f_h,  v_h)_{0,\E} &\text{if  $ \, k >4$}.
	\end{cases}
\end{equation}

We observe that the right-hand side defined above is computable by using the DoFs $\boldsymbol{D}$. 

\begin{remark}\label{Remark:L2-H1:norms}
By combining the ideas in~\cite{AABMR13} with the advancements presented here, it is possible to design enhanced 
versions of the local spaces presented in~\eqref{local:virtual:space}, allowing to compute exactly also the higher-order projection $\Pi_{\E}^{k-2} \colon \VE \to \P_{k-2}(\E)$, for $k \geq 2$.  In this case, we can  simply take $ f_h|_\E :=\Pi_{\E}^{k-2},$ for all $\E \in \O_h$ in order to derive optimal estimates in $L^2$- and $H^1$-norms.  We refer to~\cite{ChM-camwa,ABSV2016} for enhanced versions of the $C^1$-VEM with straight edges.  
\end{remark}

\paragraph{The virtual element formulation.}
With the above definitions we are in a position to write the conforming discrete 
VE formulation. The  VE problem reads as: find $u_h\in\Vh$, such that
 \begin{equation}\label{discrete:problem}
 a^h(u_h,v_h)= \langle f_h,v_h \rangle \qquad \forall v_h\in\Vh,
 \end{equation}
 where the bilinear form $a^{h}(\cdot,\cdot)$ and discrete load term $f_h$
 are defined in \eqref{global:ah} and \eqref{global:load:term}, respectively.

The well-posedness of discrete problem~\eqref{discrete:problem} follows from the fact that, by construction, the bilinear form $a^{h}(\cdot,\cdot)$ is symmetric and coercive on the space $\Vh$. This last fact will be rigorously investigated in Section~\ref{STAB:SECTION}.

We point out that if $\O$ is a domain with  polygonal boundary $\Gamma$, i.e., if $\Gamma$ is made up of a finite number of straight sides, and the stabilizing form $\mathcal{S}_{\E}(\cdot,\cdot)$ is defined as the classical ``\emph{dofi-dofi}'' stabilization, then we recover the $C^1$-VEM proposed in~\cite{BM13} apart for the choice of the stabilization. 

\begin{remark}
We again underline that the present approach (and analysis) is immediately applicable to the case of curved internal interfaces since the provided DoFs guarantee the global continuity of the discrete functions also across curved edges. Assuming curved edges only on the boundary $\Gamma$ simplifies the exposition.
\end{remark}
\setcounter{equation}{0}	
\section{Theoretical analysis}\label{THEORIC:SECTION}
This section is devoted to the interpolation, stability and convergence analysis. For a regular enough function we will construct an interpolant in the virtual element space $\Vh$ and prove optimal approximation properties. Moreover, we present stability bounds for the stabilization \eqref{eq:stab:form}. Finally, 
with the help of these previous properties, we provide the well-posedness and a priori error estimates for the proposed method.  
\subsection{Preliminary results}

We start by reviewing some preliminary results, which will be useful in the forthcoming sections. 
\begin{lemma}\label{Lemma:poli:approx}
Let $k \in \mathbb{N}$.
Then, for all $E \in \CT_h$ and $v \in H^s(E)$, there exists a  polynomial function $v_{\pi} \in \P_k(E)$,  such that
\begin{equation*}
| v - v_{\pi} |_{r, \E} \lesssim \, h_E^{s-r}| v - v_{\pi} |_{s, E} 
\lesssim \, h_E^{s-r}| v |_{s, E},
\end{equation*}
for all $r, s \in \R$ satisfying $0 \leq r \leq s \leq k+1$, where the hidden constant depend only on $k$ and the shape regularity constant $\rho$ 
	(cf. Assumption $\mathbf{(A1)}$).
\end{lemma}

We have the following trace estimate for curved domains, which is a particular case of \cite[Theorem 3.4]{BRV2019_M2AN}, 
with $\varepsilon =0$.
\begin{lemma}
\label{lemma:trace:BRV}
Under the assumptions $\mathbf{(A0)}-\mathbf{(A2)}$,  
for all $E \in \Omega_h$ and $v \in H^{1}(E)$, it holds
\begin{equation*}
| v|_{1/2, \partial E} \lesssim  |v|_{1, E} \,.
\end{equation*}
The hidden constant in the previous estimate depends only on the shape regularity constant $\rho$ and $\|\gamma'\|_{L^{\infty}}$, in particular such constant does not depend on $h_E$.
\end{lemma}

The following result establishes the extension of the scaled trace theorem over straight to curved polygons.
\begin{lemma}\label{lemma:scaled:trace}
For all $E \in \CT_h$, the following scaled trace inequalities hold
\begin{equation*}
\begin{split}
\| v \|^2_{0,\dE}  &\lesssim \, h_{\E}^{-1} \|v\|^2_{0,\E} +  h_{\E}|v|^2_{1,\E} \quad \: \: \: \forall v \in H^1(\E),\\
|v|^2_{1,\partial \E}  &\lesssim \,  h^{-1}_{\E}|v|^2_{1,\E} + |v|^2_{3/2,\E} \qquad \forall v \in H^{3/2}(\E),
\end{split}	
\end{equation*}
where the hidden constants depend only on the shape regularity constant $\rho$ and $\gamma'$ (in particular it does not depend on $h_\E$).
\end{lemma}
\begin{proof}
The proof follows from by combining the  trace and Young inequalities, along with the arguments of~\cite[Lemma 3.4]{BRV2019_M2AN} (see also \cite{BS18-M3AS}).

\end{proof}

The following estimates are a consequence of the Sobolev inequalities presented in~\cite[Eqs. (2.3) and (2.4)]{BS18-M3AS}.
\begin{align}
\|v\|_{L^{\infty}(\E)} &\lesssim h^{-1}_{\E} \|v\|_{0,\E}  + h_{\E} |v|_{2,\E}, 
\quad &\forall v \in H^2(\E), \label{Sobo-H2}\\
\|v\|_{L^{\infty}(\E)} &\lesssim h^{-1}_{\E} \|v\|_{0,\E} + h^{1/2}_{\E} |v|_{3/2,\E},
\quad &\forall v \in H^{3/2}(\E). \label{Sobo-H3-2}
\end{align}

In addition, we state the following well-known Poincar\'e--Friedrichs inequalities on Lipschitz domains. 
\begin{align}
h^{-1}_{\E} \|v\|_{0,\E} &\lesssim |v|_{1,\E} + h^{-1}_{\E}\left| \int_{\partial \E} v \, {\rm d}s\right| \qquad \qquad \qquad \qquad \qquad \forall v \in H^1(\E), \label{Poincare:ineq:H1}\\ 
h^{-2}_{\E} \|v\|_{0,\E}& \lesssim |v|_{2,\E} + h^{-2}_{\E}\left| \int_{\partial \E} v \, {\rm d}s\right| + h^{-1}_{\E}\left| \int_{\partial \E} \nabla v \, {\rm d}s\right| \qquad \forall v \in H^2(\E).\label{Poincare:ineq}
\end{align}

We close this section by recalling the properties of the Stein extension operator~$\widetilde{\mathcal{E}}$ in~\cite[Chapter~VI, Theorem $5'$]{Stein1970}. 
\begin{lemma}\label{Stein:extension}
Given a Lipschitz domain~$\Omega \subset\R^2$ and $t \in \R$ an non-negative real number, there exists an extension operator
$\widetilde{\mathcal{E}}: H^t(\Omega)\rightarrow H^t(\R^2)$ and a positive constant $C_t$ such that
	\begin{equation*}
	\widetilde{\mathcal{E}}(v)|_{\Omega} = v 
		\qquad \text{and} \qquad
		\|\widetilde{\mathcal{E}}(v)\|_{t, \R^2} \lesssim\, \|v\|_{t,\O} 
		\quad \forall v \in H^t(\Omega),
\end{equation*}
where hidden constant depends on~$t$ and on specific geometric parameters of~$\Omega$ such as Lipschitz continuity, see 
\cite[Chapter~VI, Theorem $5'$]{Stein1970}.
\end{lemma}

\subsection{Interpolation results}
We now present the following preliminary lemma, which concerns solely the approximation properties on the boundary of the elements.
\begin{lemma}\label{Lemma:edge:interp}
Let $e \subset \dE$ be an (possibly curved) edge of an element $\E \in \CT_h$ and $v$ a function in $H^s(\E)$, for some $s>2$, such that $v|_e \in H^t(e)$ and $\dn v|_e \in H^{t-1}(e)$ for some $t > 3/2$. Moreover, let $\tvI \in \Vh(\E)$ be the interpolant of $v$ with respect to the boundary DoFs $\boldsymbol{D^{\partial}}$ introduced in Section~\ref{sec:dofs}.
	Then, for all $\: 0 \leq m \leq t$ it holds
	\begin{equation}\label{approx:interp:boud}
		|v - \tvI|_{m, e} \lesssim  h_E^{t-m} \, \|v\|_{t, e}.
	\end{equation} 
Furthermore,  for all  $\: 0 \leq m \leq t-1$ we have
	\begin{equation}\label{approx:interp:dn}
		|\dn v - \dn \tvI|_{m, e} \lesssim\, h_E^{t-1-m} \, \|\dn v\|_{t-1, e}.
	\end{equation} 
	The hidden constants in the two above estimates are independent of $h_{\E}$.
	\end{lemma}
	\begin{proof}
		Note that the  result involves only the edges of the element and the associated boundary degrees of freedom. Therefore it is actually a one-dimensional approximation result for interpolants in mapped polynomial spaces. As a consequence, the proof easily follows by combining the arguments of \cite[Lemma 3.2]{BRV2019_M2AN} and the definition of degrees of freedom $\boldsymbol{D^{\partial}}$.
		
	\end{proof}	

Next theorem represents the main result of this subsection, which provides the construction of an interpolant (for a smooth enough function) in the virtual element space $\Vh$ defined in~\eqref{glob:VEMspace}. Moreover, this result establishes the corresponding error estimate.
\begin{theorem}\label{theorem:interpolation}
Under Assumptions $\mathbf{(A0)}-\mathbf{(A2)}$,  
for all $v \in H^{s}(\O)\cap V$, with $2 < s \leq k+1$, there exists $\vI \in \Vh$ such that
\begin{equation*}
 | v - \vI |_{2,\O}  \leq C \, h^{s-2} \|v\|_{s,\O},
\end{equation*}
where $C$ is a positive constant which depends on the degree $k$, the shape regularity constant $\rho$ and the parametrization $\gamma$, but is independent of $h$.
\end{theorem}
\begin{proof}
Let $v \in  H^{s}(\O) \cap V$, then for all $\E \in \CT_h$, we can choose the polynomial function 
$v_{\pi} \in \P_{k}(\E)$, such that Lemma~\ref{Lemma:poli:approx} it holds,  i.e., 
\begin{equation}\label{trian:inq:01}
|v - v_{\pi}|_{2,\E} \lesssim  \, h_E^{s-2}|v|_{s,\E}.	
\end{equation}

Let us consider the following biharmonic problem: find $\vI$ such that 
\begin{equation}\label{prob:biharmonic}
\left\{
	\begin{aligned}
& \Delta^2 \vI  =  \, \Delta^2 v_{\pi} \qquad & \text{in $E$,} \\
&  \vI =\tvI  \qquad & \text{ on $\partial E$,}\\
&  \partial_{\bn_\E} \vI =  \dn\tvI  \qquad & \text{ on $\partial E$,}
	\end{aligned}
	\right.
\end{equation}
where $\tvI$ is the interpolant in Lemma~\ref{Lemma:edge:interp} (note that, by standard trace regularity results, $v$ satisfies the hypotheses of such Lemma with $t=s-1/2$). Since $\tvI \in \VE$, the boundary data in the above problem are polynomials (or mapped polynomials) on the edges of $\E$. Furthermore, also noting that $\tvI \in \VE \subset H^2(E)$, the boundary conditions in \eqref{prob:biharmonic} clearly satisfy the matching conditions at the vertices of $\E$ required in~\cite[pp. 16--17]{GR}. Therefore, we can conclude that the solution $\vI$ of problem~\eqref{prob:biharmonic} belongs to $H^2(\E)$. Furthermore, by definition of $\tvI$ and the $H^s(\Omega)$ regularity of $v$, the traces of $\vI$ (defined as above on each element $E$) will be continuous across internal mesh edges, and the same holds for the tangential and normal components of the gradient across edges. Therefore also at the global level $\vI \in \Vh$. 

We now consider the difference $(\vI - v_{\pi})|_{\E}$, which satisfies the following local problem:
\begin{equation}\label{eq:diff:biharm}
\left\{
\begin{aligned}
& \Delta^2 (\vI  - v_{\pi})= 0 \qquad & \text{in $E$,} \\
&  (\vI - v_{\pi})=  (\tvI - v_{\pi})  \qquad & \text{on $\partial E$,}\\
&  \partial_{\bn_\E} (v_{\rm I} - v_{\pi}) = \partial_{\bn_\E} (\tvI - v_{\pi}) \qquad & \text{on $\partial E$.}
	\end{aligned}
	\right.
\end{equation}

Now, we will introduce the following notation: for $i \in \{1,3\}$, we consider the scaled norm  
$||| \cdot |||_{i/2,e}$  on $e$, defined by
\begin{equation*}
||| \varphi |||_{i/2,e} := \big(h^{-i}_{\E}\|\varphi\|^2_{0,e} + |\varphi|^2_{i/2,e}\big)^{1/2}	 
\qquad \forall \varphi \in H^{i/2}(e).
\end{equation*}
Since $(\vI  - v_{\pi})$ is biharmonic, from \cite[pp. 17]{GR}, we deduce the following (scaled) continuous dependence on the data  
for problem \eqref{eq:diff:biharm}
\begin{equation}\label{cont:dep:bihar}
|\vI  - v_{\pi}|^2_{2,\E}	\lesssim  \sum_{e \in \partial\E} (|||\partial_{\bn_\E} (\vI - v_{\pi})|||^2_{1/2,e} + |||\vI - v_{\pi}|||^2_{3/2,e} )=: \sum_{e \in \partial\E} ( \: T_{1,e} + T_{2,e}),
\end{equation}
where the hidden constant is uniform respect to the element $\E$. 

First, we will estimate the term $T_{1,e}$ in~\eqref{cont:dep:bihar}. Adding and subtracting $\dn v$, then by using triangular inequality, Lemmas~\ref{lemma:scaled:trace}, \ref{lemma:trace:BRV} and~\ref{Lemma:poli:approx}, we have 
\begin{equation}\label{eq:pre:I_e}
\begin{split}
T_{1,e}& = |||\dn(\vI - v_{\pi})|||^2_{1/2,e}
\lesssim  |||\dn (\vI  - v)|||^2_{1/2,e} + |||\dn ( v - v_{\pi})|||^2_{1/2,e} \\
& =  |||\dn ( v-\vI )|||^2_{1/2,e} + h^{-1}_\E\|\nabla ( v - v_{\pi})\|^2_{0,e}+ |\nabla ( v - v_{\pi})|^2_{1/2,e} \\
& \lesssim  |||\dn ( v-\vI )|||^2_{1/2,e} + h^{-1}_{\E}(h^{-1}_{\E}\|\nabla ( v - v_{\pi})\|^2_{0,\E}+ h_{\E}|\nabla ( v - v_{\pi})|^2_{1,\E})+ |\nabla ( v - v_{\pi})|^2_{1,\E}\\
& \lesssim  |||\dn (v-\tvI )|||^2_{1/2,e} + h_{\E}^{2(s-2)}\|v\|^2_{s,\E} ,
\end{split}	
\end{equation} 
where we also used that $\dn v_I = \dn \tvI$ on $\partial E$.

Next, we will study the first term in \eqref{eq:pre:I_e}. Let us denote $\chi_e :=\partial_{\bn_\E} (v- \tvI)|_e$; by using the Poincar\'e inequality in one dimension (note that $\chi_e$ vanishes at both endpoints of $e$) we obtain $\|\chi_e\|_{0, e}^{1/2} \lesssim h_{\E}^{1/2} |\chi_e|^{1/2}_{1,e}$. 
Thus, by combining such observation and estimates~\eqref{approx:interp:dn} (with $m=1/2$ and $t=s-1/2>3/2$),  we get
\begin{equation}\label{error:edges}
||| \chi_e |||_{1/2,e}  \lesssim |\chi_e |_{1/2,e} \lesssim |\dn v -\dn \tvI|_{1/2,e}
\lesssim h_e^{s-2} \| \dn v  \|_{s-3/2,e} \lesssim h_{\E}^{s-2} \|\nabla v\|_{s-3/2,e}. 
\end{equation}
Now, we will distinguish two cases. If the edge $e$ is straight, then we can apply a standard trace inequality to obtain
\begin{equation*}
||| \chi_e |||_{1/2,e} \lesssim h_{\E}^{s-2} \|\nabla v\|_{s-3/2,e} \lesssim h_{\E}^{s-2} \|\nabla v\|_{s-1,E} \lesssim h_{\E}^{s-2} \| v\|_{s,E}. 
\end{equation*}

If the edge $e$ is curved, then we keep the estimate~\eqref{error:edges}. Therefore, we can deduce that 
\begin{equation}\label{eq2:pre:I_e}
|||\dn (v-\tvI )|||^2_{1/2,e} \lesssim	\left\{
	\begin{aligned}
		& h^{2(s-2)}_{\E} \| v\|^2_{s,E} \qquad & \text{if $e$ is straight,} \\
		& h^{2(s-2)}_{\E} \|\nabla v\|^2_{s-3/2,e} \qquad & \text{if $e$ is curved.}
		\end{aligned}
	\right.
\end{equation} 

Thus, from estimates~\eqref{eq:pre:I_e} and~\eqref{eq2:pre:I_e}, and adding over all the edges, we can infer 
\begin{equation}\label{sum:I_e}
\sum_{e \in \partial\E} T_{1,e}  = \sum_{e \in \partial\E} |||\dn (\vI - v_{\pi})|||^2_{1/2,e} 
\lesssim h_{\E}^{2(s-2)}(\| v\|^2_{s,\E}+ \|\nabla v\|^2_{s-3/2,\dE \cap \Gamma}).
\end{equation} 

In this part, we will analyze the term $T_{2,e}$. By employing similar arguments as in $T_{1,e}$, 
we arrive
\begin{equation*}
	\begin{split}
T_{2,e}&:= |||\vI - v_{\pi}|||^2_{3/2,e}\lesssim  |||v-\vI |||^2_{3/2,e} + ||| v - v_{\pi}|||^2_{3/2,e} \\
& = |||v-\vI|||^2_{3/2,e} + h^{-3}_\E\|v - v_{\pi}\|^2_{0,e}+ | v - v_{\pi}|^2_{3/2,e} \\
& =  ||| v-\vI|||^2_{3/2,e} + h^{-3}_{\E}(h^{-1}_{\E}\|v - v_{\pi}\|^2_{0,\E}
+ h_{\E}| v - v_{\pi}|^2_{1,\E})+ |v - v_{\pi}|^2_{2,\E}\\
& \lesssim  |||v-\tvI |||^2_{3/2,e} + h_{\E}^{2(s-2)}\|v\|^2_{s,\E} .
	\end{split}	
\end{equation*} 

We observe that $(v-\tvI)|_e \in H^{3/2}(e)$ and thus, by recalling bound~\eqref{approx:interp:boud}, we obtain
\begin{equation*}
|v-\tvI|^2_{3/2,e}  \lesssim h_{\E}^{2(s-2)}\|v\|^2_{s-1/2, e}. 
\end{equation*}

Following the same arguments  as those used to obtain~\eqref{sum:I_e}, we can deduce
\begin{equation}\label{sum:II_e}
\sum_{e \in \partial\E} T_{2,e} = \sum_{e \in \partial\E} |||\vI - v_{\pi}|||_{3/2,e}^2 \lesssim h_{\E}^{2(s-2)}(\| v\|^2_{s,\E}+ \| v\|^2_{s-1/2,\dE \cap \Gamma}).
\end{equation} 

Now, we combine~\eqref{cont:dep:bihar},~\eqref{sum:I_e} and~\eqref{sum:II_e}, then summing over all the elements $\E \in \CT_h$, yields
$$
\sum_{E \in \CT_h}| v - \vI |^2_{2, E} \lesssim \, h_{\E}^{2(s-2)} \, \Big( \, \|v\|^2_{s, \O} + \sum_{i=1}^N 
(\|v \|^2_{s - 1/2, \Gamma_i} + \|\nabla v \|^2_{s - 3/2, \Gamma_i})  \Big),
$$
from where it follows
\begin{equation}\label{pre:interp}
| v - \vI |_{2, \O} \lesssim \, h^{s-2} \, \Big( \, \|v\|_{s, \O} + \sum_{i=1}^N 
(\|v \|_{s - 1/2, \Gamma_i} + \|\nabla v \|_{s - 3/2, \Gamma_i}) \Big).
\end{equation}

For any curve $\Gamma_i \subset \partial \Omega$, let $\mathcal{D}_i$ be a domain in $\R^2$, with boundary $\partial \mathcal{D}_i \in C^{s-2, 1}$, such that $\Gamma_i \subset \partial \mathcal{D}_i$. Thus, we apply the trace theorem for smooth domains~\cite{Lions-Magenes}, and by Lemma~\ref{Stein:extension},  we get
\begin{equation*}
\begin{split}
\|v\|_{s-1/2, \Gamma_i} &= \|\widetilde{\mathcal{E}}(v)\|_{s-1/2, \Gamma_i}  \leq
\|\widetilde{\mathcal{E}}(v)\|_{s-1/2, \partial \mathcal{D}_i} \leq 
\|\widetilde{\mathcal{E}}(v)\|_{s, \mathcal{D}_i} \leq
\|\widetilde{\mathcal{E}}(v)\|_{s, \R^2}  \lesssim  \|v\|_{s,\O}.
\end{split}
\end{equation*} 

Analogously, applying a vector valued version of the Stein extension
operator in Lemma~\ref{Stein:extension}, we have 
$$\|\nabla v\|_{s-3/2, \Gamma_i} = \|\widetilde{\mathcal{E}}(\nabla v)\|_{s-3/2, \Gamma_i} \lesssim  \|\nabla v\|_{s-1,\O} \lesssim  \|v\|_{s,\O}.$$ 

We obtain the desired result by combining the above estimates, the triangular inequality together with the bounds~\eqref{pre:interp} and \eqref{trian:inq:01}. Finally, we observe that the hidden constants are independent of the mesh discretization parameter $h$. This completes the proof.

\end{proof}

\subsection{Stability analysis}\label{STAB:SECTION}
In this subsection, we establish a stability bound for the stabilizing form $\mathcal{S}_E(\cdot,\cdot)$ introduced in~\eqref{eq:stab:form}. To begin, we observe that, by its definition, this form can be extended to the space $H^{5/2}(\E)$ for each $E \in \O_h$.

The following result will play a crucial role in the subsequent sections, particularly in proving the coercivity of the bilinear form $a^h(\cdot,\cdot)$ (see Proposition~\ref{prop:stability:coerc}).

\begin{proposition}\label{prop:stability:S_E}
Let~$\E \in \O_h$ and  $\mathcal{S}_E(\cdot,\cdot)$ be the bilinear form defined in~\eqref{eq:stab:form}. Then, under assumptions $\mathbf{(A0)}-\mathbf{(A2)}$, for all $w \in H^{5/2}(\E)$, such that both $w$ and $\nabla w$  have zero average on $\partial \E$, the following inequality holds:
\begin{equation}\label{continuity:S_E}
	\mathcal{S}_E(w,w)\lesssim |w|^2_{2,\E}+h_{\E}|w|^2_{5/2,\E},
	\end{equation}
where the hidden constant is independent of $h_{\E}$.
\end{proposition}

\begin{proof}
From definition~\eqref{eq:stab:form}, we have
\begin{equation}\label{split:I}
	\begin{split}
		\mathcal{S}_E (w,w) 
		&= \mathcal{S}^{\boldsymbol{\circ}}_{\E} (w, w) + \mathcal{S}^{\boldsymbol{\partial},\vb}_{\E} (w, w)  + \mathcal{S}^{\boldsymbol{\partial},e}_{\E} (w, w)\\
		&\approx h^{-2}_{\E}  \sum_{j=1}^{(k-3)(k-2)/2} \boldsymbol{D}^{\boldsymbol{\circ}}_j(w)^2
		+ \sum_{i=1}^{3N_{\E} } h^{-2}_{\vb_i} \boldsymbol{D}^{\boldsymbol{\partial},\vb}_i(w)^2  
		+ h_{\E} \sum_{i=1}^{N_e} |\nabla w|^2_{1,e}. 
	\end{split}
\end{equation}

First, we will analyze the term $\mathcal{S}_{\E}^{\boldsymbol{\circ}}(w,w)$. Indeed,  by recalling that the polynomial basis  $\{m_j\}_{j=1}^{(k-3)(k-2)/2}$ in $\boldsymbol{D^{\circ}}$ satisfies $\|m_j\|_{L^{\infty}(E)} \leq 1$, and using the 
Cauchy-Schwarz inequality we obtain
\begin{equation}\label{eq:S^o}
	\begin{split}
		\mathcal{S}_{\E}^{\boldsymbol{\circ}}(w,w)
		&=\sum_{j=1}^{(k-3)(k-2)/2} h^{-2}_E \left( h^{-2}_E \int_E w \, m_j \, {\rm d}E  \right)^2 	\\
		&\leq h^{-2}_E h^{-4}_E \left(\int_E w \,\, {\rm d}E  \right)^2 \, \sum_{j=1}^{(k-3)(k-2)/2} \|m_j\|_{L^{\infty}(E)}^2 \\
		& \lesssim h^{-2}_E h^{-4}_E |E| \|w\|_{0,E}^2 
		\lesssim  h^{-4}_E\|w\|_{0,E}^2.
	\end{split}
\end{equation}

Now, we will relabel the index for the term $\mathcal{S}^{\boldsymbol{\partial},\vb}(w_h,w_h)$ (see DoFs $\boldsymbol{D^{\partial}_{\!I}}$ and $\boldsymbol{D^{\partial}_{\!I\!I}}$), as follows
\begin{equation}\label{eq:S^partial}
	\begin{split}
		\mathcal{S}_{E}^{\boldsymbol{\partial},\vb}(w,w)
		&=\sum_{i_1=1}^{N_{\E}} h^{-2}_{\vb_i}\boldsymbol{D}^{\boldsymbol{\partial, \: \!I}}_{i_1}(w)^2 + \sum_{i_2=1}^{2N_{\E}} h^{-2}_{\vb_i}\boldsymbol{D}^{\boldsymbol{\partial,\: \!I\!I}}_{i_2}(w)^2
		=:\mathcal{S}_{\E}^{\boldsymbol{\partial,\: \!I}}(w,w) +\mathcal{S}_{\E}^{\boldsymbol{\partial,\: \!I\!I}}(w,w).
	\end{split}
\end{equation}

By using the inequality~\eqref{Sobo-H2} and  $ h_{\vb_{i_1}} \approx h_{\E}$, we have 
\begin{equation}\label{eq:S^partial:1}
	\begin{split}
		\mathcal{S}_{\E}^{\boldsymbol{\partial, \: \!I}}( w, w) &\lesssim   N_E \:h^{-2}_{\vb_{i_1}}\| w\|^2_{L^{\infty}(\dE)} \lesssim h^{-2}_{\E} \| w\|^2_{L^{\infty}(\E)}
		\lesssim  h^{-4}_{\E} \| w\|^2_{0,\E} + | w|^2_{2,\E}.
	\end{split}
\end{equation} 

For the term $\mathcal{S}^{\boldsymbol{\partial, \: \!I\!I}}( w, w)$,  we use the inequality~\eqref{Sobo-H3-2} (applied to $\nabla w$), to obtain
\begin{equation}\label{eq:S^partial:2}
	\begin{split}
		\mathcal{S}_{E}^{\boldsymbol{\partial, \: \!I\!I}}( w, w)
		&\lesssim 2N_E  h^{-2}_{\vb_{i_1}}  h^{2}_{\vb_{i}}\|\nabla  w\|^2_{L^{\infty}(\E)}
		\lesssim  | w|^2_{2,\E} +  h_{E}| w|^2_{5/2,\E}.
	\end{split}
\end{equation}

Now, by using the second trace inequality of Lemma~\ref{lemma:scaled:trace}, for the remaining term we have  
\begin{equation}\label{eq:S^partial:3}
	\mathcal{S}_{E}^{\boldsymbol{\partial},e}( w, w) \lesssim	 h_{\E} \sum_{i=1}^{N_e} |\nabla  w|^2_{1,e}
	\lesssim |\nabla  w|^2_{1,\E}  + h_{\E} |\nabla  w|^2_{3/2,\E}  \lesssim | w|^2_{2,\E}  + h_{\E} | w|^2_{5/2,\E}.	
\end{equation}

Thus, by combining  \eqref{eq:S^o}-\eqref{eq:S^partial:3} and \eqref{split:I}, we get
\begin{equation*}	
\calS_E( w, w)\lesssim h_E^{-4} \| w\|^2_{0,\E} + | w|_{2,\E}^2+h_{\E}| w|_{5/2,\E}^2.	
\end{equation*}

Finally, the bound~\eqref{continuity:S_E} follows by combining the above estimate and the Poincar\'e--Friedrichs inequality on Lipschitz domains~\eqref{Poincare:ineq}. The proof is complete.

\end{proof}

Now, we recall the following result, which establishes a norm equivalence on polygons (see for instance~\cite{BRV2019_M2AN}). 

\begin{lemma}\label{lemma:l2:norm:equi}
Let $E \in \Omega_h$. Under the assumptions $\mathbf{(A0)}-\mathbf{(A2)}$,
let $\mathbf{g}:= (g_j)_{j=1}^{d_t}$ be a vector of real coefficients and let $g := \sum_j^{d_t} g_j \, m_j \in \P_{t}(E)$, 
where $d_t={\rm dim}(\P_{t}(\E))$. Then, the following norm equivalence holds
\[
 h_E^2 \, \|\mathbf{g}\|^2_{\ell^2} \lesssim \|g\|^2_{0, E} \lesssim \, h_E^2 \, \|\mathbf{g}\|^2_{\ell^2},
\]
where the hidden constants are uniform and $\|\cdot\|_{\ell^2}$ is the classical $\ell^2$-norm.
\end{lemma}

Additionally, we have the following $H^2$-orthogonality decomposition.
\begin{lemma}\label{H2:decomp}
Any function $v_h \in H^2(\E)$ admits the decomposition $v_h= v_1 +v_2$, where 
\begin{itemize}
	\item $v_1 \in  H^2(\E)$, $v_1|_{\partial \E} = v_h|_{\partial \E}$, $\partial_{\bn_{\E}}v_1 = \partial_{\bn_{\E}} v_h$  
	and $\Delta^2 v_1 =0 \qin \E$;
	\item $v_2 \in  H_0^2(\E)$, $\Delta^2 v_2= \Delta^2 v_h \qin \E$.
\end{itemize}
Moreover, this decomposition is $H^2$-orthogonal in the sense that
\begin{equation}\label{H2:orthonal}
	|v_h|^2_{2,\E}= |v_1|^2_{2,\E} +|v_2|^2_{2,\E}.
\end{equation}
\end{lemma}

\begin{proof}
Let $v_h $ be an element of $H^2(\E)$. Then, we can choose $v_2$  as the $H^2$-projection of $v_h$ in $H_0^2(\E)$ and define
$v_1:= v_h-v_2 \in H^2(\E)$.  We observe that by construction the functions $v_1$ and $v_2$ satisfy the properties of the lemma.

\end{proof}

We now introduce the following definition (for all mesh elements $E$). A function $v$ is said to satisfy assumption ${\cal H}_t(\E)$, for some $t \ge 2$, if $v \in H^2(\E)$ and, for every edge $e \subset \partial \E$, the inclusions $v|_e \in H^t(e)$ and $\dn v|_e \in H^{t-1}(e)$ hold.
\begin{remark}\label{remark:H_t}
From the definition of the virtual space $\VE$, we know that its functions are regular on the edges and satisfy the appropriate compatibility conditions at the vertices. Therefore, the functions in $\VE$ satisfy assumption $\mathcal{H}_t(\E)$,  for some  $t \ge 2$ depending on the  curved edge parametrization regularity (which, we recall, it is at least $C^{k+1}$). 
Furthermore, functions in the space $\VE + \P_1(\E)$ also satisfy this assumption (again, thanks to the regularity of the edges). This observation will be useful in the forthcoming analysis, in particular for Lemma~\ref{lemma:coercivity:S_E}.
\end{remark} 

Now we observe that for any function $w$ satisfying assumption $\mathcal{H}_t(\E)$, for some $t \ge 2$, the following Poincar\'e-type estimate in one-dimension is easy to check
\begin{equation*}
	\|w\|^2_{0,e} \lesssim h^2_{\E}|w|^2_{1,e} +h_{\E} |w(\vb_e^{{\rm ex}})|^2 \qquad \text{and} \qquad 	
	|w|^2_{1,e} \lesssim h^2_{\E}|w|^2_{2,e} +h_{\E} |\nabla w(\vb_e^{{\rm ex}})|^2,	
\end{equation*}
where $\vb_e^{{\rm ex}}$ is an extremal point of $e$ (note that here $\nabla w(\vb_e^{{\rm ex}})$ is well defined since both the tangent and normal derivatives along the edge are in $H^1$). 
Thus, we can deduce that
\begin{equation*}
	\begin{split}
		h^{-3}_{\E}\|w\|^2_{0,e} \lesssim h^{-1}_{\E}|w|^2_{1,e} +h^{-2}_{\E} |w(\vb_e^{{\rm ex}})|^2 \qquad \text{and} \qquad 	
		h^{-1}_{\E}|w|^2_{1,e} \lesssim h_{\E}|w|^2_{2,e} + |\nabla w(\vb_e^{{\rm ex}})|^2,	
	\end{split}
\end{equation*}
which implies 
\begin{equation*}
		\begin{split}
h^{-3}_{\E}\|w\|^2_{0,e} + h^{-1}_{\E}|w|^2_{1,e} &\lesssim h_{\E}|w|^2_{2,e} + |\nabla w(\vb_e^{{\rm ex}})|^2 + h^{-2}_{\E} |w(\vb_e^{{\rm ex}})|^2.
	\end{split}
\end{equation*}

Moreover, by analogous arguments
\begin{equation*}
h^{-1}_{\E}\|\partial_{\bn_{\E}} w\|^2_{0,e} \lesssim 
h_{\E} |\partial_{\bn_{\E}} w|^2_{1,e} + |\nabla w(\vb_e^{{\rm ex}})|^2 .
\end{equation*}

From the above bounds and the definition of  $\mathcal{S}_{\E}(\cdot,\cdot)$ we can conclude that for all function $w$ satisfying assumption $\mathcal{H}_t(\E)$ (for some $t \ge 2$) it holds
\begin{equation}\label{pre:coerc:S_E}
\mathcal{S}_E(w,w) \gtrsim \sum_{e \in \partial\E} \Big( h^{-1}_{\E}\|\partial_{\bn_{\E}} w\|^2_{0,e} + h_{\E} |\partial_{\bn_{\E}} w|^2_{1,e}  + h^{-3}_{\E}\|w\|^2_{0,e} + h^{-1}_{\E}|w|^2_{1,e} + h_{\E}|w|^2_{2,e} \Big).  
	\end{equation} 

The above property will be fundamental to prove the coercivity of the stabilizing form $\calS_E(\cdot,\cdot)$ in the space $\VE+ \P_1(\E)$, which is established in the following result.

\begin{lemma}\label{lemma:coercivity:S_E}
Under Assumptions $\mathbf{(A0)}-\mathbf{(A2)}$ it holds
\begin{equation*}
\mathcal{S}_E(v_h,v_h) \gtrsim  |v_h|^2_{2,\E} \qquad \forall v_h \in \VE + \P_1(\E),  \quad \text{and} \quad   \forall \E \in \CT_h,
\end{equation*}
where the hidden constant is independent of $h_E$.
\end{lemma}
\begin{proof}
Let $v_h \in  \VE + \P_1(\E) \subset H^2(\E)$ and $v_1 \in H^2(\E)$ and $v_2 \in H_0^2(\E)$ such that Lemma~\ref{H2:decomp} holds true.  
First, we will show that $|v_1|^2_{2,\E} \lesssim \calS_E(v_h,v_h)$.  Indeed, as in~\eqref{cont:dep:bihar} we use the continuous dependence of the data  and the construction of  $v_1$,
as follows
\begin{equation}\label{cont:dep:II}
|v_1|^2_{2,\E} \lesssim 
\sum_{e \in \dE} (|||\partial_{\bn_E} v_h|||^2_{1/2,e}+|||v_h|||^2_{3/2,e})=: \sum_{e \in \dE} (T_{1,e}+T_{2,e}).
\end{equation}  

Using property~\eqref{pre:coerc:S_E} and Remark~\ref{remark:H_t}, we immediately obtain
\begin{equation*}
h^{-1}_{\E}	\|\partial_{\bn_E} v_h\|^2_{0,e} \lesssim  \calS_E(v_h,v_h).
\end{equation*}

Now, by employing the real interpolation method, the Young inequality, and again property~\eqref{pre:coerc:S_E}, we get
\begin{equation*}
\begin{split}
	|\partial_{\bn_E} v_h|^2_{1/2,e} \lesssim \|\partial_{\bn_E} v_h\|_{0,e} |\partial_{\bn_E} v_h|_{1,e}
	\lesssim  h^{-1}_{\E} \|\partial_{\bn_E} v_h\|^2_{0,e}  +h_{\E} |\partial_{\bn_E} v_h|^2_{1,e}
	 \lesssim \calS_E(v_h,v_h).
\end{split}
\end{equation*}

Therefore, for the term $T_{1,e}$, we have the following bound
\begin{equation}\label{I_e:stab}
T_{1,e} = |||\partial_{\bn_E} v_h|||^2_{1/2,e} = h^{-1}_{\E} \|\partial_{\bn_E} v_h\|^2_{0,e}  +  |\partial_{\bn_E} v_h|^2_{1/2,e}\lesssim   \calS_E(v_h,v_h).
\end{equation}

Analogously, by using property~\eqref{pre:coerc:S_E} and Remark~\ref{remark:H_t}, we can derive  
\begin{equation*}
h^{-3}_{\E} \|v_h \|^2_{0,e} \lesssim \calS_E(v_h,v_h) \quad \text{and} \quad 
|v_h|^2_{3/2,e} \lesssim |v_h|_{1,e} |v_h|_{2,e} \lesssim  \calS_E(v_h,v_h) .
\end{equation*}

Thus, by combining the above estimates we have 
\begin{equation}\label{II_e:stab}
T_{2,e} = ||| v_h|||^2_{3/2,e} = h^{-3}_{\E} \| v_h\|^2_{0,e}  +  | v_h|^2_{3/2,e}
\lesssim   \calS_E(v_h,v_h).
\end{equation}
Inserting~\eqref{I_e:stab} and \eqref{II_e:stab} in~\eqref{cont:dep:II}, we conclude 
\begin{equation}\label{bound:v_1}
|v_1|^2_{2,\E} \lesssim  \calS_E(v_h,v_h).
\end{equation}

Now, we will analyze the second part of the $H^2$-decomposition. Indeed, let $g:=\Delta^2 v_h=\Delta^2 v_2 \in \P_{k-4}(\E)$, then  using the fact that $v_h=v_1+v_2$ and an integration by parts, it follows that
\begin{equation}\label{ident:integration}
|v_2|^2_{2,\E} = \int_{\E} v_2\Delta^2 v_2 \:{\rm d}\E=  \int_{\E} g v_2 \:{\rm d}\E
=  \int_{\E} g v_h \: {\rm d}\E- \int_{\E} g v_1 \:{\rm d}\E=: T_{1,\E}+ T_{2,\E}.
\end{equation}
 
Since $g \in \P_{k-4}(\E)$, then we can write 
\begin{equation*}
g =\sum_{j=1}^{(k-3)(k-2)/2}  g_j  m_j.
\end{equation*}

Then, from the definition of $\boldsymbol{D^{\circ}}$, the Cauchy-Schwarz inequality for sequences, Lemma~\ref{lemma:l2:norm:equi}, the first term can be bounded as follows
\begin{equation*}
\begin{split}
T_{1,\E} &=  \int_{\E} gv_h  \: {\rm d}\E =  \sum_{j=1}^{(k-3)(k-2)/2}  g_j \int_{\E} m_j v_h\: {\rm d}\E 
=  \sum_{j=1}^{(k-3)(k-2)/2}  h^{2}_E \,  g_j  \, \boldsymbol{D}_j^{\boldsymbol{\circ}}(v_h)\\
&\lesssim h^3_{\E} \Bigg(\sum_{j=1}^{(k-3)(k-2)/2}  g^2_j\Bigg)^{1/2}  \Bigg(\sum_{j=1}^{(k-3)(k-2)/2} h^{-2}_{\E} \boldsymbol{D}_j^{\boldsymbol{\circ}}(v_h)^2\Bigg)^{1/2}\\
&  \lesssim h^3_{\E} \|\gb\|_{\ell^2}\, \calS^{\boldsymbol{\circ}}_E(v_h,v_h)^{1/2} \\
& \lesssim h^2_{\E} \|\Delta^2 v_2\|_{0,\E}\,\calS_E(v_h,v_h)^{1/2}.
	\end{split}
\end{equation*}

We now observe that, exploiting the fact that $\Delta^2 v_2$ is a polynomial function in $\E$, we can apply arguments similar to those in \cite[Lemma 6.3]{BLR:stab}, leading to
\begin{equation}\label{xx:new}
\|\Delta^2 v_2\|_{0,\E} \lesssim  h^{-2}_{\E} |v_2|_{2,\E} .
\end{equation}
Thus, 
\begin{equation}\label{eq:I:v_h}
T_{1,\E} \lesssim 	|v_2|_{2,\E} \,\calS_E(v_h,v_h)^{1/2}.
\end{equation}

On the other hand, by applying again the Cauchy-Schwarz and~\eqref{xx:new}  
we have
\begin{equation}\label{pre:eq:II:v_h}
\begin{split}
T_{2,\E} =  \int_{\E} gv_1 \: {\rm d}\E   \leq \|\Delta^2 v_2\|_{0,\E} \|v_1\|_{0,\E}  \lesssim 
h^{-2}_{\E} |v_2|_{2,\E}\|v_1\|_{0,\E}. 	
\end{split}
\end{equation}

The goal now is to show that $h^{-2}_{\E} \|v_1\|_{0,\E} \lesssim \calS_E(v_h,v_h)^{1/2}$. 
To this end, recalling that on $\partial E$ it holds $v_h=v_1$, $\dn v_h = \dn v_1$, 
we first start applying the Poincar\'e--Friedrichs inequality~\eqref{Poincare:ineq} and H\"{o}lder inequality. We then proceed using~\eqref{pre:coerc:S_E} and \eqref{bound:v_1} to obtain
\begin{equation*}
\begin{split}
\|v_1\|_{0,\E} & \lesssim h^2_{\E} |v_1|_{2,\E} + \left|\int_{\partial \E } v_1 \, {\rm d}s \right| + h_{\E}\left|\int_{\partial \E } \nabla v_1  \, {\rm d}s \right|\\
& \lesssim h^2_{\E} |v_1|_{2,\E} +h^{1/2}_{\E} \|v_h\|_{0,\partial\E}
+ h^{3/2}_{\E} |v_h|_{1,\partial\E}\\
& \lesssim h^2_{\E} |v_1|_{2,\E} + h^2_{\E}\,\calS_E(v_h,v_h)^{1/2}\\
 & \lesssim h^2_{\E}\,\calS_E(v_h,v_h)^{1/2}.
\end{split}
\end{equation*}

Then, from  the above bound and~\eqref{pre:eq:II:v_h}, we conclude
\begin{equation*}
T_{2,\E}  \lesssim |v_2|_{2,\E} \,\calS_E(v_h,v_h)^{1/2}.
\end{equation*}

Therefore, by combining~\eqref{eq:I:v_h}, the above estimate and~\eqref{ident:integration}, it follows that
\begin{equation}\label{bound:v_2}
|v_2|_{2,\E}^2 \lesssim \calS_E(v_h,v_h).
\end{equation} 

Finally, the desired result follows by inserting estimates~\eqref{bound:v_1} and~\eqref{bound:v_2} into~\eqref{H2:orthonal}. 

\end{proof}

We have the following stability result.
\begin{proposition}\label{prop:stability:coerc}
There exist a positive uniform constants $\alpha_{*}$ such that for any element $E \in \CT_h$ it holds that
\begin{equation*}
		a_E^h (v_h, \, v_h) \geq \alpha_* \, a_E (v_h, \, v_h) \, \qquad \forall v_h \in \VE.
\end{equation*}
As a consequence the global bilinear form $a^h(\cdot,\cdot)$ is coercive in $\Vh$.
\end{proposition}
\begin{proof}
In order to simplify the following notation, we temporarily denote by ${\overline\Pi}_E : = \Pi_E^{{\bf D},1}$, which is the operator defined in \eqref{Ritz:operator} for the particular case $k=1$. Such operator, acting from $\VE$ onto ${\mathbb P}_1(E)$, preserves by construction the boundary integral of the function and its gradient.

Let $v_h \in \VE$, then we set $\widetilde{v}:= (v_h-\overline\Pi_E v_h) \in \VE + \P_1(\E)$. First, we note that from definitions~\eqref{Ritz:operator} and~\eqref{averag:operator}, we deduce  $(I-\PiK)\widetilde{v}=(I-\PiK)v_h$. Now,
by employing Lemma~\ref{lemma:coercivity:S_E}, some algebraic manipulations and the above observation, we obtain
\begin{equation}\label{pre:coerc:I}
\begin{split}
a_E(v_h,v_h) &= a_E(\widetilde{v},\widetilde{v})  \lesssim \calS_E(\widetilde{v},\widetilde{v}) 	 \lesssim \calS_E(\PiK\widetilde{v},\PiK\widetilde{v}) +\calS_E((I-\PiK)\widetilde{v},(I-\PiK)\widetilde{v}) \\ 
& = \calS_E(\PiK\widetilde{v},\PiK\widetilde{v}) +\calS_E((I-\PiK)v_h,(I-\PiK)v_h).
\end{split}
\end{equation} 

Again, from definitions~\eqref{Ritz:operator} and~\eqref{averag:operator} it follows that 
$ {\overline\Pi}_E (\PiK\widetilde{v}) = 
{\overline\Pi}_E \widetilde{v}  = 0$,
thus, by using property~\eqref{continuity:S_E} of Proposition~\ref{prop:stability:S_E},  we get
\begin{equation*}
	\calS_E(\PiK\widetilde{v},\PiK\widetilde{v}) \lesssim |\PiK\widetilde{v}|^2_{2,\E}	+ h_{\E}  |\PiK\widetilde{v}|^2_{5/2,\E}.
\end{equation*}

Now, we observe that inserting the above bound in~\eqref{pre:coerc:I} we can infer
\begin{equation}\label{pre:coerc:II}
a_E(v_h,v_h) \lesssim |\PiK\widetilde{v}|^2_{2,\E}	+ h_{\E}  |\PiK\widetilde{v}|^2_{5/2,\E}+\calS_E((I-\PiK)v_h,(I-\PiK)v_h).
\end{equation}

By using standard polynomial inverse inequality on star-shaped polygons we have
\begin{equation*}
h_{\E} |\PiK\widetilde{v}|^2_{5/2,\E} \lesssim |\PiK\widetilde{v}|^2_{2,\E}.	
\end{equation*}

Moreover,  from definition of operator $\PiK$ in \eqref{Ritz:operator}, 
we easily observe that $|\PiK\widetilde{v}|^2_{2,\E} =|\PiK v_h|^2_{2,\E}$. 
Therefore, by combining the above facts and~\eqref{pre:coerc:II}, we obtain the desired result.
\end{proof}

\begin{remark}\label{remark:error:estimate}
We recall that, differently from the Poisson problem in \cite{BRV2019_M2AN}, the space $\VE$ does not contain the kernel of $a_{\E}(\cdot,\cdot)$, i.e., $\P_1(\E) \nsubseteq \VE$ (see Remark~\ref{remark:non:poly:containing} and definition~\eqref{Ritz:operator}). However, due to  Lemma~\ref{lemma:coercivity:S_E} it is still possible to obtain the first inequality in estimate \eqref{pre:coerc:I}, hence the coercivity of the discrete bilinear form 
$a^h_{\E}(\cdot,\cdot)$.  Some intuitive explanation on the importance of the kernel can be found also in Remark~\ref{Remark:new:stab}.
The aforementioned fact is our main motivation for introducing the new stabilization term $\mathcal{S}_{\E}(\cdot, \cdot)$ (cf. \eqref{eq:new:stab:form}).
\end{remark}
%
\subsection{A priori error  estimates}
In this subsection we will provide an a priori error analysis for our conforming virtual element scheme. We start by recalling a bound for the load approximation error. We omit the simple proof since it follows from standard VEM arguments, see for instance~\cite{BM13}.

\begin{proposition}\label{func-bound}
	Let $2 < s \leq k +1$ and 
	\begin{equation}\label{def:alpha}
		\alpha =\begin{cases}
			0   &\text{if  $k =2,3$},\\
			1  &\text{if  $k =4$},\\
			k-3  &\text{if  $k >4$ and $s>3$.}	
		\end{cases}
	\end{equation}
Then, under assumptions ${\bf (A1)}$-${\bf (A2)}$ and $f \in  H^{\alpha}(\O)$, we have the following estimate
	\begin{equation*}
		\|f-f_h\|:= \sup_{v_h \in \Vh} \frac{|(f,v_h)_{0,\O}-\langle f_h, v_h \rangle|}{|v_h|_{2,\O}}\lesssim h^{s-2}|f|_{\alpha,\O}.  
	\end{equation*}
\end{proposition}	

For all $v \in H^s(\Omega)$ with $s > 2$, we introduce the following quantity:
\begin{equation}\label{eq:non-conformity}
		T^h(v) := \Big(\: \sum_{E \in \O_h} a^h_E(v_{\pi} - v_{\rm I}, \,v_{\pi} - v_{\rm I}) \: \Big)^{1/2},	 
\end{equation}	
where $v_{\rm I} \in \Vh$ and $v_{\pi} \in \P_k(\O_h)$ are the VEM interpolant and polynomial projection of $v$ in the sense of Theorem~\ref{theorem:interpolation} and Lemma~\ref{Lemma:poli:approx}, respectively. 

\begin{remark}\label{remark:non-conformity}
The term  $T^h (\cdot)$ defined in~\eqref{eq:non-conformity} will appear in the abstract error estimate below and is related to the VEM approximation of the bilinear form of the problem.
A difficulty in bounding this term lies in the fact that we cannot directly use standard arguments (see for instance~\cite{BRV2019_M2AN}), since the norm appearing on the right-hand side of the continuity bound~\eqref{continuity:S_E} is stronger than the norm in which we derived approximation estimates; compare Theorem~\ref{theorem:interpolation} and Proposition~\ref{prop:stability:S_E}. Furthermore, since this term involves the sum space 
$\Vh + \P_k(\O_h)$, it is not easily handled by bridging the two norms with inverse inequalities. We provide an error bound for $T^h(u)$ in Lemma~\ref{lemma:non-coformity}. 
\end{remark}

The following result establishes an error estimate for the  consistency term defined in~\eqref{eq:non-conformity}. 
\begin{lemma}\label{lemma:non-coformity}
Let $u \in V$ be  the solution of  problem~\eqref{Bihar:weak} and assume that $u \in H^{s}(\O)$, 
with $5/2 \leq s \leq k +1$. Then, the following estimate holds 
\begin{equation*}
			T^h(u) \lesssim h^{s-2} \|u\|_{s,\O}.
\end{equation*}
\end{lemma}
\begin{proof}
By using the definition of $a^h_E(\cdot,\cdot)$ and $\calS_E(\cdot,\cdot)$ we obtain
\begin{equation*}
	\begin{split}
		a^h_E(u_{\pi} - u_{\rm I}, \,u_{\pi} - u_{\rm I}) & \lesssim |u_{\pi} - u_{\rm I}|_{2,\E}^2 
		+ \calS_E((I-\PiK)(u_{\pi} - u_{\rm I}), (I-\PiK)(u_{\pi} - u_{\rm I}))\\
		&= |u_{\pi} - u_{\rm I}|_{2,\E}^2 
		+ \calS^{\boldsymbol{\circ}}_E((I-\PiK) u_{\rm I}, (I-\PiK) u_{\rm I})
		+\calS^{\boldsymbol{\partial}}_E((I-\PiK) u_{\rm I}, (I-\PiK) u_{\rm I}).
	\end{split}
\end{equation*}

For simplicity, we denote $\overline{u}_{{\rm I}}:=(I-\PiK) u_{\rm I}$, then from
\eqref{eq:S^o}, \eqref{eq:S^partial} and \eqref{eq:S^partial:1}, we have 
\begin{equation*}
	\begin{split}
		\calS^{\boldsymbol{\circ}}_E((I-\PiK) u_{\rm I}, (I-\PiK) u_{\rm I})
		&+\calS^{\boldsymbol{\partial}}_E((I-\PiK) u_{\rm I},(I-\PiK) u_{\rm I})  \\
		&\leq h_{\E}^{-4}\|\overline{u}_{{\rm I}}\|^2_{0,\E} + h_{\E}^{-2}|\overline{u}_{{\rm I}}|^2_{1,\E} 
		+  |\overline{u}_{{\rm I}}|^2_{2,\E}+\calS^{\boldsymbol{\partial, \; \!I\!I}}_E(\overline{u}_{{\rm I}}, \overline{u}_{{\rm I}})
		+ \calS^{\boldsymbol{\partial},e}_E(\overline{u}_{{\rm I}}, \overline{u}_{{\rm I}}).
	\end{split}
\end{equation*}

By definition $\Pi_{\dE}^{0}\overline{u}_{{\rm I}} =\Pi_{\dE}^{0}(\nabla\overline{u}_{{\rm I}}) =0$,  thus  by combining the two above estimates and the Poincar\'e--Friedrichs inequality (cf.~\eqref{Poincare:ineq}), it follows that
\begin{equation}\label{main:eq}
	\begin{split}
		a^h_E(u_{\pi} - u_{\rm I}, \,u_{\pi} - u_{\rm I}) &\lesssim |u_{\pi} - u_{\rm I}|^2_{2,\E} + |\overline{u}_{{\rm I}}|^2_{2,\E}+ \calS^{\boldsymbol{\partial, \; \!I\!I}}_E(\overline{u}_{{\rm I}}, \overline{u}_{{\rm I}})
		+ \calS^{\boldsymbol{\partial},e}_E(\overline{u}_{{\rm I}}, \overline{u}_{{\rm I}})\\
		&=: T^E_1 + T^E_2 + T^E_3+ T^E_4.
	\end{split}	
\end{equation}
By using the triangle inequality, for the first term, we get
\begin{equation}\label{eq:T1}
	T_1^E \lesssim |u_{\pi} - u|^2_{2,E} + |u - u_{\rm I}|^2_{2,E}.
\end{equation}
Concerning the term $T_2^E$, by the continuity of the operator $\PiK$ we obtain
\begin{equation}\label{eq:T2}
	\begin{split}
		T_2^E &=|(I-\PiK) u_{\rm I}|^2_{2,E} \lesssim |(I -\PiK)( u - u_{\rm I})|_{2, E}^2 + |(I -\PiK ) u|_{2, E}^2 \\
&		\lesssim  
		|u - u_{\rm I}|^2_{2,E}+|u-u_{\pi} |^2_{2,E}.
	\end{split}
\end{equation}

Now, we will analyze the term $T^E_3$. By definition we have 
\begin{equation*}
	\begin{split}
		T^E_3 &=\sum_{i_2=1}^{2N_E} h^{-2}_{i_2}\boldsymbol{D}^{\boldsymbol{\partial,  \; \!I\!I}}_{i_2}(\overline{u}_{{\rm I}})^2
		=\sum_{i_2=1}^{2N_E} h^{-2}_{i_2}\boldsymbol{D}^{\boldsymbol{\partial,  \; \!I\!I}}_{i_2}(u-\PiK u_{{\rm I}})^2
		\lesssim   \|\nabla (u-\PiK u_{{\rm I}})\|^2_{L^{\infty}(\E)}\\ 
&		\lesssim \|\nabla(u-\PiK u)\|^2_{L^{\infty}(\E)}+ \|\nabla\PiK (u- u_{{\rm I}})\|_{L^{\infty}(\E)}=: T^E_{3,1}+T^E_{3,2}.
	\end{split}
\end{equation*}

By adding and subtracting $u_{\pi}$, then using the estimate~\eqref{Sobo-H3-2} (applied to $\nabla(u-u_{\pi})$), together with standard inverse inequality for polynomials on star-shaped polygons, we deduce
\begin{equation*}
	\begin{split}
		T^E_{3,1} &\leq  \|\nabla(u-u_{\pi})\|^2_{L^{\infty}(\E)} 
		+ \|\nabla\PiK(u_{\pi}- u)\|^2_{L^{\infty}(\E)}\\
		& \lesssim h_{\E}^{-2} |u-u_{\pi}|^2_{1,\E}+ h_{\E} |u-u_{\pi}|^2_{5/2,\E} 
		+ h_{\E}^{-2}\|\nabla\PiK(u_{\pi}- u)\|^2_{0,\E}\\
		& \lesssim|u-u_{\pi}|^2_{2,\E}+ h_{\E} |u-u_{\pi}|^2_{5/2,\E} +|\PiK(u_{\pi}- u)|^2_{2,\E}\\
		& \lesssim|u-u_{\pi}|^2_{2,\E}+ h_{\E} |u-u_{\pi}|^2_{5/2,\E},
	\end{split}
\end{equation*}
where we also used Lemma~\ref{Lemma:poli:approx}, and the continuity  of projection $\PiK$ respect to the seminorm $|\cdot|_{2,\E}$ in the second and last step, respectively.

Next, we will bound term $T^E_{3,2}$.  By using the inverse and Poincar\'e--Friedrichs inequality~\eqref{Poincare:ineq:H1}, we obtain
\begin{equation*}
	\begin{split}
		T^E_{3,2} &= \|\nabla\PiK (u-u_{{\rm I}})\|^2_{L^{\infty}(\E)}  \lesssim h^{-2}_{E}\|\nabla\PiK (u-u_{{\rm I}})\|^2_{0,\E}\\
		& \lesssim h^{-2}_{E}\|\nabla\PiK (u-u_{{\rm I}})-\nabla(u-u_{{\rm I}})\|^2_{0,\E}+h^{-2}_{E}\|\nabla(u-u_{{\rm I}})\|^2_{0,\E}\\
		& \lesssim |\nabla\PiK (u-u_{{\rm I}})-\nabla(u-u_{{\rm I}})|^2_{1,\E} + h^{-2}_{E}\|\nabla(u-u_{{\rm I}})\|^2_{0,\E}\\
		& \lesssim  | u-u_{{\rm I}}|^2_{2,\E} + h^{-2}_{E}\|\nabla(u-u_{{\rm I}})\|^2_{0,\E},
	\end{split}
\end{equation*}
where we have used also the continuity of projection $\PiK$ respect to the seminorm $|\cdot|_{2,\E}$. 

Now, by using again the Poincar\'e--Friedrichs inequality \eqref{Poincare:ineq:H1}, splitting the gradient in tangent and normal components (and recalling that the integral on each edge of the tangent derivative of $u-u_{{\rm I}}$ vanishes since such function is zero at all vertexes) we have
\begin{equation*}
	\begin{split}
		h^{-1}_{E}\|\nabla(u-u_{{\rm I}})\|_{0,\E}	
		&\lesssim  |\nabla(u-u_{{\rm I}})|_{1,\E} + h^{-1}_{\E}\left| \int_{\partial \E} \nabla(u-u_{{\rm I}}) \, {\rm d}s\right| \\
		&\lesssim  |\nabla(u-u_{{\rm I}})|_{1,\E} + h^{-1}_{\E}\left| \int_{\partial \E} \dn(u-u_{{\rm I}}) \, {\rm d}s\right| \\
		&\lesssim  |u-u_{{\rm I}}|_{2,\E} + h^{-1/2}_{\E} \| \dn(u-u_{{\rm I}}) \|_{0,\partial \E} .
	\end{split}
\end{equation*} 

By employing Lemma~\ref{Lemma:edge:interp},  we have that 
\begin{equation*}
	h^{-1/2}_{\E} \| \dn(u-u_{{\rm I}}) \|_{0,e}
	\lesssim h^{s-2}_{\E}\| \dn u\|_{s-3/2,e},
\end{equation*}
and thus we easily derive, by the same arguments used in Theorem \ref{theorem:interpolation}, 
\begin{equation*}
	T^E_{3,2}
	\lesssim | u-u_{{\rm I}}|^2_{2,\E}+ h_{\E}^{2(s-2)} \| \nabla u \|^2_{s-3/2,\partial \E \cap \Gamma} .		
\end{equation*}  
By collecting the  bounds involving $T^E_{3,1}$ and $T^E_{3,2}$, we conclude
\begin{equation}\label{eq:T3}
	T^E_3\lesssim  h_{\E} |u-u_{\pi}|^2_{5/2,\E} +|u-u_{\pi}|^2_{2,\E} + | u-u_{{\rm I}}|^2_{2,\E}+ h_{\E}^{2(s-2)} \| \nabla u \|^2_{s-3/2,\partial \E \cap \Gamma} .
\end{equation}

By adding and subtracting suitable terms, we have that 
	\begin{equation}\label{split:T_4}
		\begin{split}
			T^E_4 \lesssim \sum_{e \in \partial \E} h_{E} |\nabla\overline{u}_{{\rm I}}|^2_{1,e} 
			&\lesssim h_{E} \sum_{e \in \partial \E}  \Big( |\nabla(u_{{\rm I}}- u_{\pi})|^2_{1,e} + |\nabla(u_{\pi}- \PiK u)|^2_{1,e} + |\nabla\PiK( u- u_{{\rm I}})|^2_{1,e} \Big) \\  
			&=:h_{E} \sum_{e \in \partial \E}  (T^e_{4,1} +T^e_{4,2} +T^e_{4,3}). 		
		\end{split}
\end{equation}

For the term $T^e_{4,1}$, we proceed applying the second trace inequality in Lemma~\ref{lemma:scaled:trace}, as follow 
	\begin{equation}\label{T_1^e}
		\begin{split}
			T^e_{4,1} &\lesssim |\nabla(u_{{\rm I}}- u)|^2_{1,e} + |\nabla(u- u_{\pi})|^2_{1,e} \\
			&  \lesssim |u_{{\rm I}}- u|^2_{2,e} + h^{-1}_{\E} |\nabla(u- u_{\pi})|^2_{1,\E} + |\nabla(u- u_{\pi})|^2_{3/2,\E} \\      
			&\lesssim   h^{2(s-5/2)}\|u\|^2_{s-1/2,e} +h^{-1}_{\E} |u- u_{\pi}|^2_{2,\E}+  |u- u_{\pi}|^2_{5/2,\E},
		\end{split}
	\end{equation}
	where we have applied Lemmas~\ref{Lemma:edge:interp} in the first term.

Next, for the term $T^e_{4,2}$, we apply again the scaled trace bound  in Lemma~\ref{lemma:scaled:trace}, to obtain 
	\begin{equation*}
		\begin{split}
			T^e_{4,2} &=|\nabla (u_{\pi}-\PiK u)|^2_{1,e}  \lesssim h^{-1}_{\E}|\nabla (u_{\pi}-\PiK u)|^2_{1,\E}+ |\nabla (u_{\pi}-\PiK u)|^2_{3/2,\E}\\
			& \lesssim h^{-1}_{\E}|\PiK(u_{\pi}- u)|^2_{2,\E}+ |\PiK(u_{\pi}- u)|^2_{5/2,\E}.
		\end{split}
	\end{equation*}
	From the above estimate, we observe that if $k=2$, then  $|\PiK(u_{\pi}- u)|^2_{5/2,\E}=0$. However, if $k \geq 3$, then we can apply a classical inverse inequality for polynomials on star-shaped polygons, to obtain $ |\PiK(u_{\pi}- u)|^2_{5/2,\E} \lesssim h^{-1}_{\E} |\PiK(u_{\pi}- u)|^2_{2,\E}$. Therefore, by using the continuity of the projection $\PiK$  respect with $|\cdot|_{2,\E}$, for both cases, we get 
	\begin{equation}\label{T_2^e}
		T^e_{4,2}	 \lesssim h^{-1}_{\E}|u_{\pi}- u|^2_{2,\E}.
\end{equation}

Following similar arguments we can derive
	\begin{equation}\label{T_3^e}
		T^e_{4,3}	 \lesssim h^{-1}_{\E}|u-u_{{\rm I}}|^2_{2,\E}.
	\end{equation}

By inserting \eqref{T_1^e}-\eqref{T_3^e} into~\eqref{split:T_4}, we obtain
	\begin{equation}\label{eq:T4}
		\begin{split}
			T^{\E}_4 &\lesssim  |u-u_{\pi}|^2_{2,\E} + | u-u_{{\rm I}}|^2_{2,\E} +  h_{\E} |u- u_{\pi}|^2_{5/2,\E} +   h_{\E}^{2(s-2)} \| u \|^2_{s-1/2,\partial \E \cap \Gamma} .
		\end{split}
	\end{equation}

By combining estimates~\eqref{eq:T1},~\eqref{eq:T2},~\eqref{eq:T3},~\eqref{eq:T4}, and~\eqref{main:eq}, it follows that
\begin{equation*}
	\begin{split}
		a^h_E(u_{\pi} - u_{\rm I}, \,u_{\pi} - u_{\rm I}) &\lesssim h_{\E} 	|u- u_{\pi}|^2_{5/2,\E} 
		+|u-u_{\pi}|^2_{2,\E}   
+ | u-u_{{\rm I}}|^2_{2,\E}		+ h_{\E}^{2(s-2)} \| u \|^2_{s-1/2,\partial \E \cap \Gamma}.
	\end{split}	
\end{equation*}

Finally, the desired result follows easily from definition \eqref{eq:non-conformity} and
the above estimate, by applying similar arguments to those used
in the proof of Theorem~\ref{theorem:interpolation}, the Stein extension operator and the approximation properties in Lemma~\ref{Lemma:poli:approx}, and Theorem~\ref{theorem:interpolation}.

\end{proof}

In the following result we state an abstract error analysis in the energy norm for
scheme~\eqref{discrete:problem}, which can be seen as a Strang-type lemma and allow us to obtain the respective optimal error estimates. 


\begin{theorem}\label{theorem:convergence}
		Under assumptions $\mathbf{(A0)}$-$\mathbf{(A2)}$, let $u\in V$ and $u_h \in \Vh$ 
		be the solutions of  problems~\eqref{Bihar:weak} and \eqref{discrete:problem}, respectively. 
		Then, for all $u_{\rm I} \in \Vh$ and  each $u_{\pi} \in \P_k(\O_h)$, we have 
		\begin{equation}\label{bound:Strang-type}
			|u - u_h|_{2,\O} \lesssim |u - u_{\rm I}|_{2,\O} + |u-u_{\pi}|_{2,h}+ \|f-f_h\| + T^h(u).
		\end{equation}
		
		Furthermore, if $u \in H^{s}(\O)$ and $f \in H^{\alpha}(\O)$, with $5/2 \leq s \leq k +1$ and $\alpha$ as in \eqref{def:alpha}. Then, we have the following error estimate 
			\begin{equation*}
				|u - u_h|_{2,\O} \lesssim h^{s-2} (\|u\|_{s,\O} + |f|_{\alpha,\O}),
			\end{equation*}
		where the hidden constants depends on the degree $k$, the parametrization $\gamma$ and 
			the shape regularity constant $\rho$, but is  independent of $h$.
\end{theorem}
\begin{proof}
Let  $u_{\rm I} \in \Vh$ and $u_{\pi} \in \P_k(\O_h)$, then we set $\delta_h := (u_h - u_{\rm I}) \in \Vh$. By following standard steps in VEM analysis \cite{BBCMMR2013}, we have
	\begin{equation*}
		\begin{split}
			a^h(\delta_h, \, \delta_h) 
			&=a^h(u_h - u_{\rm I}, \, \delta_h)  =    (f_h - f, \, \delta_h) + 
			\sum_{E \in \Omega_h} \left( a^h_E(u_{\pi} - u_{\rm I}, \, \delta_h) + a_E(u - u_{\pi}, \, \delta_h) \right) \\
			& \leq C\|f-f_h\| |\delta_h|_{2,\O} + \sum_{E \in \CT_h}  a^h_E(u_{\pi} - u_{\rm I}, \, \delta_h) 
			+ |u - u_{\pi}|_{2,h}|\delta_h|_{2,\O} . 
		\end{split}
	\end{equation*}
	
	Now, since $a_E^h(\cdot,\cdot)$ is symmetric and applying the Cauchy-Schwarz and Young inequalities we get 
	\begin{equation*}
		\begin{split}
			\sum_{E \in \CT_h}  a^h_E(u_{\pi} - u_{\rm I}, \, \delta_h)  
			&\leq  \sum_{E \in \CT_h}  a^h_E(u_{\pi} - u_{\rm I}, \,u_{\pi} - u_{\rm I})^{1/2}a^h_E(\delta_h, \delta_h)^{1/2} \\
			& \leq \sum_{E \in \CT_h} \frac{1}{2}  a^h_E(u_{\pi} - u_{\rm I}, \,u_{\pi} - u_{\rm I}) + \frac{1}{2} a^h(\delta_h, \delta_h).
		\end{split}
	\end{equation*}
	
	Thus, from the above bounds, the Young inequality and Proposition~\ref{prop:stability:coerc} we easily obtain
	\begin{equation*}
	|\delta_h|^2_{2,\O} \lesssim \|f-f_h\|^2  + |u - u_{\pi}|^2_{2,h} + [T^h(u)]^2.
	\end{equation*}
	
	Therefore the desired result follows from the above bound and the triangular inequality.

	The proof of the second part follows by bounding the terms on the right-hand side of~\eqref{bound:Strang-type}, by means of Theorem~\ref{theorem:interpolation}, Lemma~\ref{Lemma:poli:approx}, Proposition~\ref{func-bound} and  Lemma~\ref{lemma:non-coformity}.
	
\end{proof}

\setcounter{equation}{0}	
\section{Numerical experiments}\label{SECTION:NUMERICAL:RESULT}
In this section we report numerical examples to test the performance of the new $C^1$-VE scheme with curved edges. In particular, we study the convergence of the proposed method from the practical perspective, comparing with Theorem~\ref{theorem:convergence}. 
Moreover, we check numerically that an approximation of the domain by using  ``straight VE'' polygons, leads to a sub-optimal rate of convergence for $k \geq 3$, as expected. 

In order to compute the experimental errors between the exact solution $u$ and the VEM solution $u_h$, we consider the following computable error quantities:
\begin{equation}\label{computable:error}
	\begin{split}
	{\rm Err}_{i}(u) &:=	\frac{\left( \sum_{E \in \Omega_h} | u -   \PiK u_h |_{i,\E}^2 \right)^{1/2}}{|u|_{i,\O}}, \quad  \text{for} \quad i =1,2, \\
	{\rm Err}_{0} (u)&:= 
	\frac{\left( \sum_{E \in \Omega_h}  \|  u -  \PiK u_h  \|_{0,E}^2 \right)^{1/2}}{\|u\|_{0,\O}} \, ,
\end{split}
\end{equation}
which converge with the same rate as the exact errors $|u - u_h|_{i,\O}$ (with $i=1,2$) and  $ \|u -  u_h\|_{0,\O}$, respectively.

We consider  the curved domain $\Omega$ described by
\begin{equation}\label{eq:test1domain}
	\Omega := \{ (x, \, y) \in \R^2 :  \quad 0 < x< 1, \quad \text{and} \quad g_{{\it bt}}(x) < y < g_{{\it tp}}(x) \} \,,
\end{equation}
where $g_{{\it bt}}$ and $g_{{\it tp}}$ are given by
\[
g_{{\it bt}}(x) := \frac{1}{20} \sin(\pi x) 
\qquad \text{and} \qquad
g_{{\it tp}}(x) := 1 + \frac{1}{20} \sin(3\pi x), \,
\]
see Figure~\ref{fig:test1domain}.  We assume that the curved boundaries $\Gamma_{{\it bt}}$ and $\Gamma_{{\it tp}}$ are parameterized with the standard graph parametrization, which are given by (see~\cite{BRV2019_M2AN}),
\[
\begin{aligned}
	\gamma_{{\it bt}} &\colon [0, \, 1] \to \Gamma_{{\it bt}} 
	&\qquad \gamma_{{\it bt}}(t)& = \left(t, \,  \frac{1}{20} \sin(\pi t)\right) ,
	\\
	\gamma_{{\it tp}} &\colon [0, \, 1] \to \Gamma_{{\it tp}} 
	&\qquad \gamma_{{\it tp}}(t)& = \left(t, \,  1 + \frac{1}{20} \sin(3 \pi t)\right).
\end{aligned}
\]
\begin{figure}[h!]
	\begin{center}
		{\includegraphics[height=7.0cm, width=8.5cm]{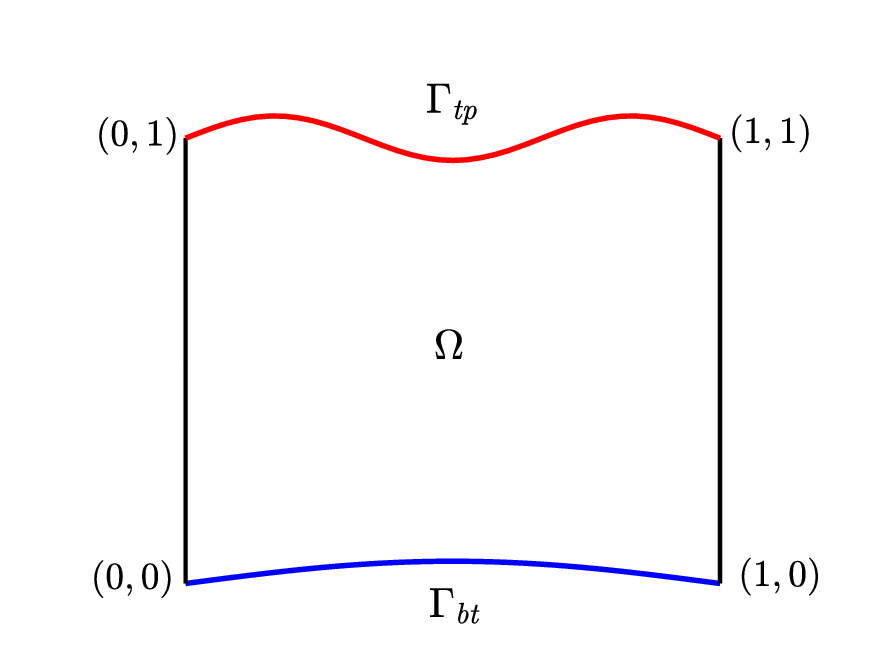}}
	\end{center}
	\caption{The domain $\Omega$ described in~\eqref{eq:test1domain}.}
	\label{fig:test1domain} 
\end{figure}

\begin{figure}[h!]
	\begin{center}
		{\includegraphics[height=6.5cm, width=8.0cm]{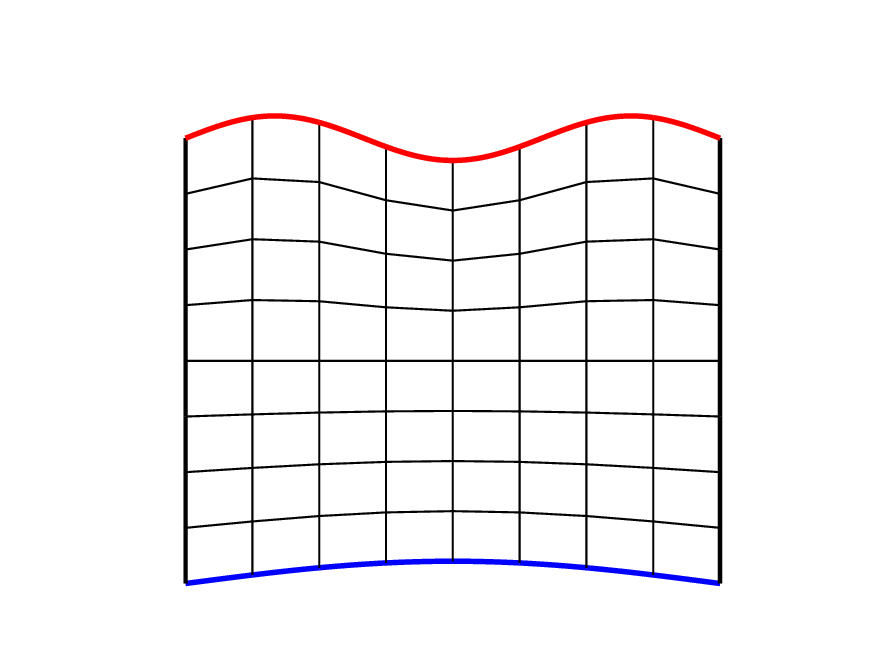}} \hspace*{-16.0mm}
		{\includegraphics[height=6.5cm, width=8.0cm]{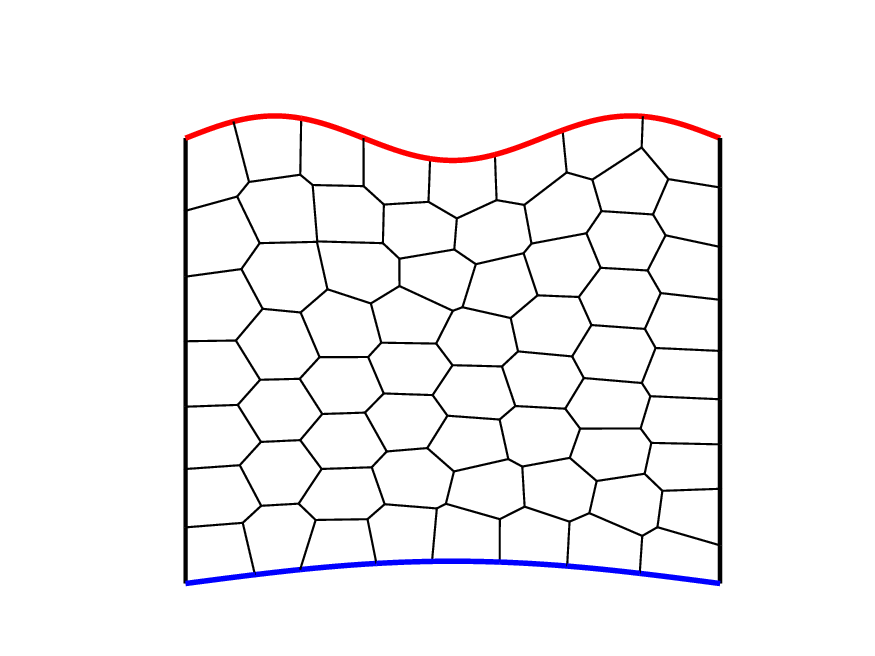}}
	\end{center}
	\caption{Example of the adopted curved polygonal meshes: quadrilateral mesh
		and  Voronoi mesh (on the left and right), respectively.}
	\label{fig:test1:domain:mesh} 
\end{figure}

In this experiment we test the Biharmonic problem~\eqref{Bihar:weak} on the curved domain $\Omega$.  We choose the load term $f$ and the homogeneous Dirichlet boundary conditions in such a way that the analytical solution is  given by 
\begin{equation*}
	u(x, \, y) = -(y - g_{{\it bt}}(x))^2(y - g_{{\it tp}}(x))^2 x^2(1-x)^2 (3 + \sin(5x) \sin(7y)).
\end{equation*}

The main objective of this test is to check the performance of the proposed VEM for curved domains, in comparison with the standard VEM in straight domains. The polygonal partition on the curved domain $\O$ is constructed starting from a mesh for the unitary square $\Omega_{Q} =(0,1)^2$ and mapping the nodes accordingly to the following rule:
\begin{equation*}
		(x_{\Omega},y_\Omega)=
		\begin{cases}
			(x_s,y_s+g_{{\it bt}}(x_s)(1-2y_s)),   &\text{if}~ y_s \leq \frac12 ,\\
			(x_s, 1-y_s+g_{{\it tp}}(x_s)(2y_s-1), & \text{if}~ y_s > \frac12,
		\end{cases}
\end{equation*}
where $(x_Q, \, y_Q)$ and $(x_{\Omega}, \, y_{\Omega})$ denote the mesh nodes on the square domain $\Omega_Q$ and on the curved domain $\Omega$,  respectively. The edges on the curved boundary consist of an arc of the curves $\Gamma_{{\it bt}}$ and $\Gamma_{{\it tp}}$  (bottom and top, respectively). In Figure~\ref{fig:test1:domain:mesh},
we depict a (curved) square mesh and a (curved) Voronoi  tessellation.

In Figure~\ref{fig:errors:k3}, we show the rate of convergence with the computable errors defined in~\eqref{computable:error} on the given sequences of meshes with uniform mesh refinements for the polynomial degree $k=3$. 
We notice that the optimal rate of convergence predicted in Theorem~\ref{theorem:convergence} is attained. 
Moreover, the expected optimal rate of convergence in the weaker $L^2$- and $H^1$-norms (i.e., orders $\mathcal{O}(h^4)$ and $\mathcal{O}(h^3)$, respectively) are also attained. 
Note that we have not proved the estimate for these cases, but the optimal rate of convergence for these weaker norms can be derived by combining the tools presented here, with a duality argument as in~\cite[Sections 4 and 5]{ChM-camwa} (see Remark~\ref{Remark:L2-H1:norms}).

\begin{figure}[h!]
	\begin{center}
		{\includegraphics[height=6.0cm, width=5.2cm]{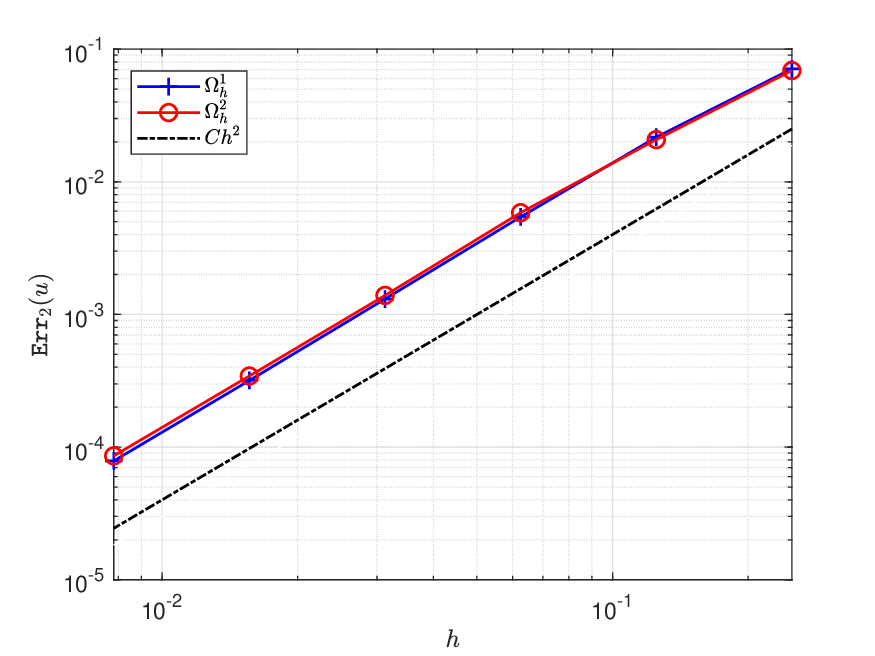}} \hspace*{-2.5mm}
		{\includegraphics[height=6.0cm, width=5.2cm]{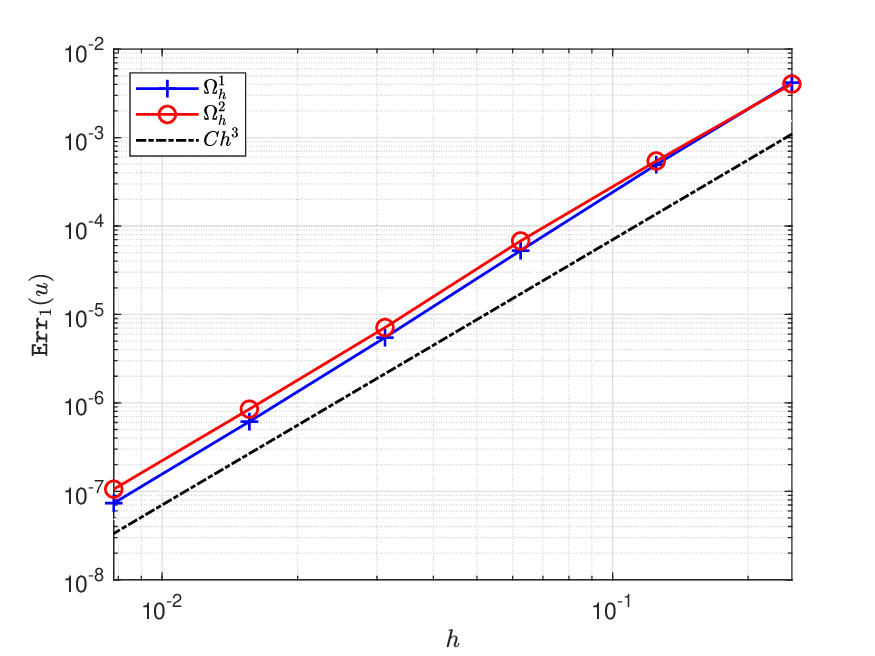}} \hspace*{-2.5mm}
		{\includegraphics[height=6.0cm, width=5.2cm]{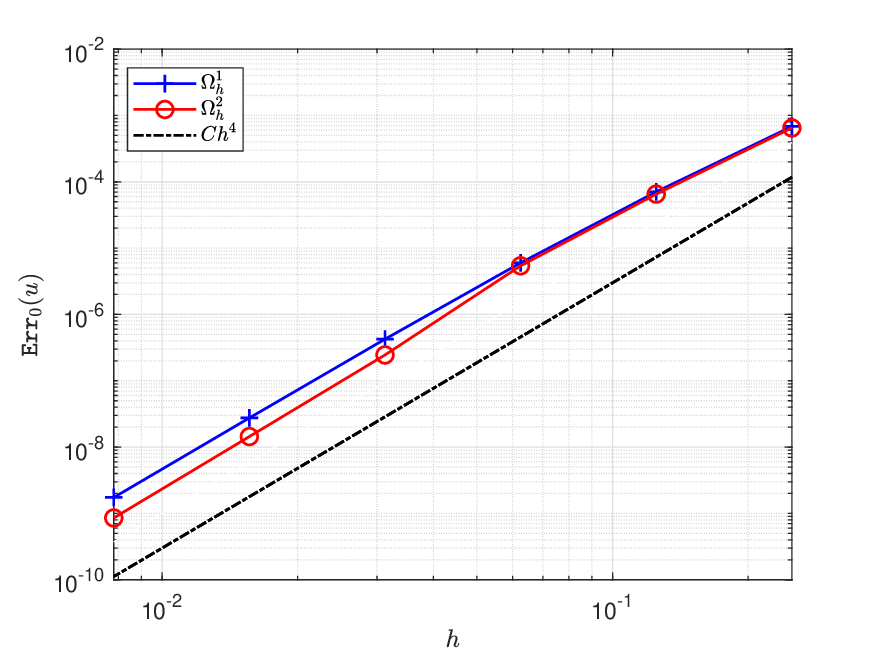}}
	\end{center}
	\caption{Errors ${\rm Err}_{2}(u)$, ${\rm Err}_{1}(u)$ and ${\rm Err}_{0}(u)$ for the  quadrilateral (curved) and Voronoi (curved)  meshes, with $k=3$.}
	\label{fig:errors:k3} 
\end{figure}

As an additional test, we approximate the curved domain by employing a Voronoi tessellation sequence
of elements with straight edges and  we force the homogeneous Dirichlet boundary conditions; see Figure~\ref{fig:stright:voronoi}. 

\begin{figure}[h!]
	\begin{center}
		{\includegraphics[height=7.0cm, width=6.8cm]{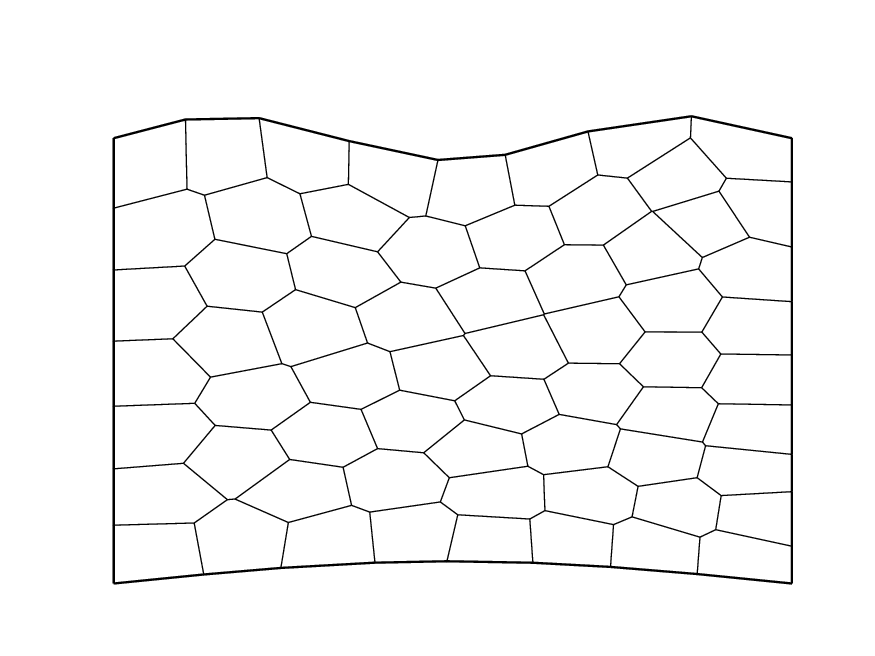}} 
	\end{center}
	\caption{An example of (straight) Voronoi tessellation over the domain $\O$.}
\label{fig:stright:voronoi} 
\end{figure}

In Figure~\ref{fig:errors_straight:k3} we show the error obtained for the sequence of Voronoi meshes on the approximated domain, by using the standard $C^1$-VEM on (straight) polygons. 

As expected, the method obtained by the approximation of the domain with straight edge polygons suffers a loss of convergence order. We note that this loss becomes very notable for the $L^2$- and $H^1$-norms (which are only of order $\mathcal{O}(h^2)$), while for the $H^2$-norm can be observed a slight loss of convergence rate in the last level of refinements.
In order to observe an appreciable difference also in the $H^2$ norm, additional experiments with $k \ge 4$ would be needed; this is beyond the current computational resources and the scopes of the present foundational contribution.

\begin{figure}[h!]
	\begin{center}
		{\includegraphics[height=6.0cm, width=5.2cm]{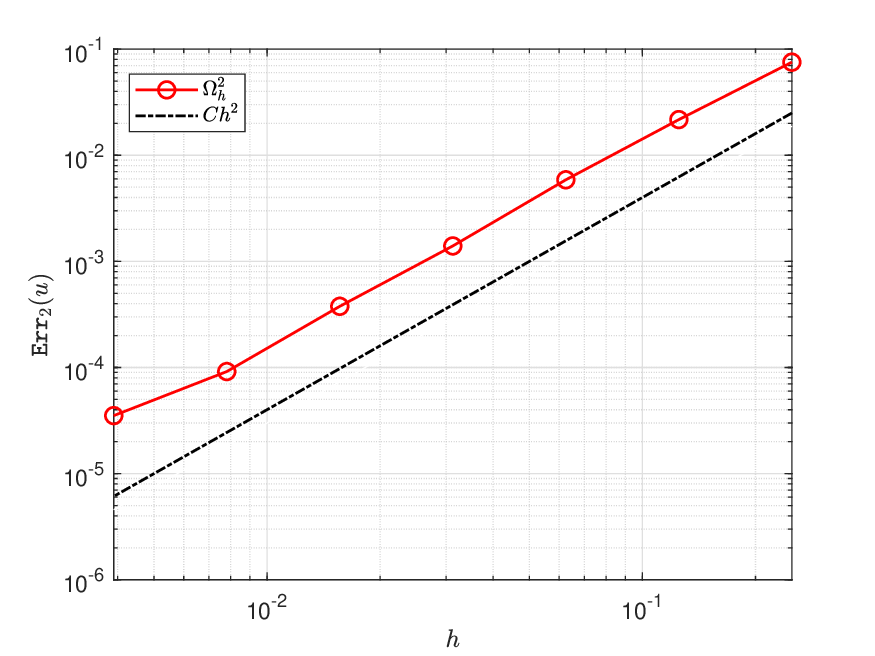}} \hspace*{-2.5mm}
		{\includegraphics[height=6.0cm, width=5.2cm]{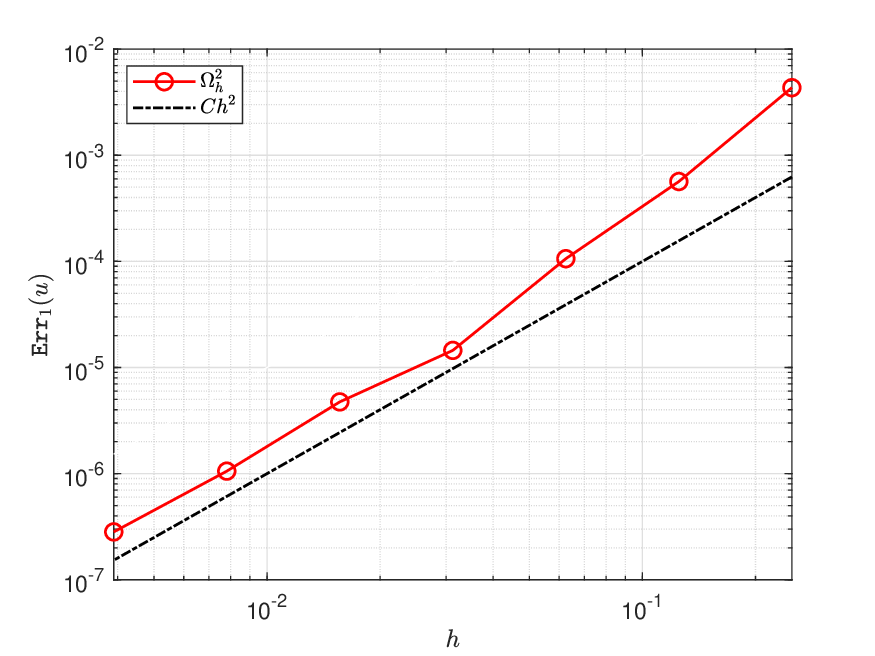}} \hspace*{-2.5mm}
		{\includegraphics[height=6.0cm, width=5.2cm]{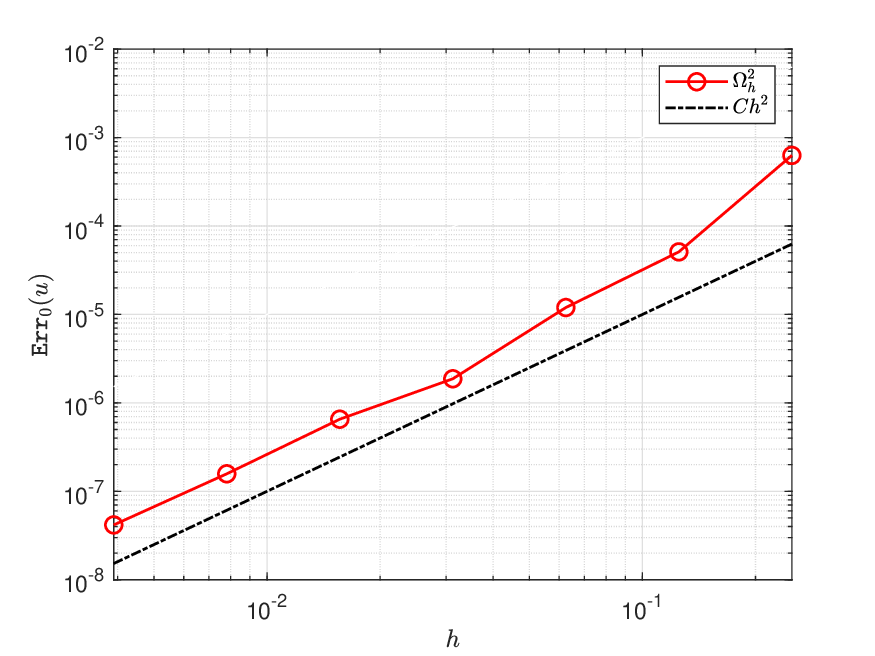}}
	\end{center}
	\caption{Errors ${\rm Err}_{2}(u)$, ${\rm Err}_{1}(u)$ and ${\rm Err}_{0}(u)$ for the Voronoi (straight) mesh, with $k=3$.}
	\label{fig:errors_straight:k3} 
\end{figure}

\small
\section*{Acknowledgements} 	
The authors are deeply grateful to Professor Alessandro Russo (Universit\'a degli studi Milano-Bicocca), for the help in the programming issues considering curved edges.
The first author was partially supported by the European Union through the ERC Synergy grant NEMESIS 
(project number 101115663).
The second author was partially supported by the National Agency
for Research and Development, ANID-Chile through FONDECYT project 1220881,
by project {\sc Anillo of Computational Mathematics for Desalination Processes} ACT210087,
and by project Centro de Modelamiento Matem\'atico (CMM), FB210005, BASAL funds for centers of excellence.
The third author was supported by project Centro de Modelamiento Matem\'atico (CMM), FB210005,
BASAL funds for centers of excellence.	

\end{document}